\documentclass[11pt]{amsart}
\usepackage{amsmath,amssymb}

\title{Hausdorff hyperspaces of $\Rm$ and their dense subspaces}

\author[W. Kubi\'s]{Wies{\l}aw Kubi\'s}
 \address[W. Kubi\'s]{Instytut Matematyki,
 Akademia \'Swi\c etokrzyska,
 25-406 Kielce, Poland}
 \email{wkubis@pu.kielce.pl}

\author[K. Sakai]{Katsuro Sakai}
 \address[K. Sakai]{Institute of Mathematics,
 University of Tsukuba, Tsukuba, 305-8571, Japan}
 \email{sakaiktr@sakura.cc.tsukuba.ac.jp}


\thanks{This research was supported by Grant-in-Aid for Scientific Reserch
 (No.\ 17540061), Japan Society for the Promotion of Science}

\newtheorem{tw}{Theorem}[section]
\newtheorem{wn}[tw]{Corollary}
\newtheorem{lm}[tw]{Lemma}
\newtheorem{prop}[tw]{Proposition}
\newtheorem{fact}[tw]{Fact}

\theoremstyle{definition}

\theoremstyle{remark}

\newtheorem{question}{Question}

\newenvironment{alphenume}{%
 \begin{enumerate}%
 }{\end{enumerate}}
\newenvironment{romanenume}{%
 \begin{enumerate}%
 }{\end{enumerate}}

\newcommand{\eps}{\varepsilon}
\renewcommand{\phi}{\varphi}
\renewcommand{\rho}{\varrho}

\newcommand{\Cee}{\mathcal{C}}
\newcommand{\Dee}{\mathcal{D}}
\newcommand{\Ef}{\mathcal{F}}
\newcommand{\Ha}{\mathcal H}
\newcommand{\Yu}{\mathcal{U}}
\newcommand{\Vee}{\mathcal{V}}

\newcommand{\Emm}{\mathfrak{M}}
\newcommand{\En}{\mathfrak N}

\newcommand{\Qyu}{\mathbb{Q}}

\newcommand{\rest}{\restriction}
\newcommand{\ntr}{n\in\omega}
\newcommand{\loe}{\leqslant}
\newcommand{\goe}{\geqslant}
\newcommand{\subs}{\subseteq}
\newcommand{\sups}{\supseteq}
\newcommand{\nnempty}{\ne\emptyset}
\newcommand{\ovr}{\overline}
\newcommand{\til}{\tilde}

\newcommand{\cl}{\operatorname{cl}}

\newcommand{\bd}{\operatorname{bd}}
\newcommand{\diam}{\operatorname{diam}}
\newcommand{\dist}{\operatorname{dist}}
\newcommand{\card}{\operatorname{card}}

\newcommand{\id}{\operatorname{id}}
\newcommand{\pr}{\operatorname{pr}}
\newcommand{\topiso}{\approx}

\newcommand{\setof}[2]{\{#1\colon #2\}}
\providecommand{\ciag}[1]{\ensuremath{\setof{{#1}_n}{\ntr}}}
\newcommand{\sn}[1]{\{#1\}} 
\newcommand{\pair}[2]{\langle #1, #2 \rangle} 
\newcommand{\seq}[1]{\langle #1 \rangle} 
\newcommand{\map}[3]{#1\colon #2 \to #3} 
\newcommand{\img}[2]{#1[#2]} 
\newcommand{\fin}[1]{[#1]^{<\omega}}

\providecommand{\nat}{\omega}

\newcommand{\R}{\ensuremath{\mathbb R}}
\newcommand{\Rm}{\ensuremath{\mathbb R^m}}

\newcommand{\bal}{\operatorname{B}}
\newcommand{\clbal}{\overline{\bal}}

\newcommand{\norm}[1]{\|#1\|}

\newcommand{\cld}{\operatorname{Cld}}
\newcommand{\bcl}{\operatorname{Bd}}
\newcommand{\nwd}{\operatorname{Nwd}}
\newcommand{\uclb}{\operatorname{UNb}}
\newcommand{\reg}{\operatorname{Reg}}
\newcommand{\perf}{\operatorname{Perf}}
\newcommand{\Cantor}{\operatorname{Cantor}}
\newcommand{\poly}{\operatorname{Pol}}
\newcommand{\Fin}{\operatorname{Fin}}
\newcommand{\effros}{\mathfrak E}
\newcommand{\borel}{\mathfrak B}

\newcommand{\Z}{\ensuremath{\operatorname{Z}}} 
\newcommand{\Zsig}{\ensuremath{\Z_\sigma}} 
\newcommand{\Q}{\ensuremath{\operatorname{Q}}} 
\newcommand{\Qmp}{\ensuremath{\Q\setminus 0}} 
\newcommand{\pseudobd}{\ensuremath{\Sigma}} 
\newcommand{\pseudoint}{\ensuremath{\operatorname{s}}} 
\newcommand{\pseudointmp}{\ensuremath{\pseudoint\setminus0}} 

\begin{document}

\begin{abstract}
Let $\bcl_H(\Rm)$ be the hyperspace of nonempty bounded closed subsets
 of Euclidean space $\Rm$ endowed with the Hausdorff metric.
It is well known that $\bcl_H(\Rm)$ is homeomorphic to
 the Hilbert cube minus a point.
We prove that natural dense subspaces of $\bcl_H(\Rm)$
 of all nowhere dense closed sets, of all perfect sets, of all Cantor sets
 and of all Lebesgue measure zero sets are homeomorphic to
 the Hilbert space $\ell_2$.
For each $0 \loe 1 < m$, let
$$\nu^m_k
 = \{x = (x_i)_{i=1}^m \in \Rm : x_i \in \R\setminus\Qyu
 \text{ except for at most $k$ many $i$}\},$$
where $\nu^{2k+1}_k$ is the $k$-dimensional N{\"o}beling space
 and $\nu^m_0 = (\R\setminus\Qyu)^m$.
It is also proved that the spaces $\bcl_H(\nu^1_0)$
 and $\bcl_H(\nu^m_k)$, $0\loe k<m-1$, are homeomorphic to $\ell_2$.
Moreover, we investigate the hyperspace $\cld_H(\R)$
 of all nonempty closed subsets of the real line $\R$
 with the Hausdorff (infinite-valued) metric.
It is shown that a nonseparable component $\Ha$ of $\cld_H(\R)$ is homeomorphic
 to the Hilbert space $\ell_2(2^{\aleph_0})$ of weight $2^{\aleph_0}$
 in case where $\Ha \not\ni \R, [0,\infty), (-\infty,0]$.
\end{abstract}

\maketitle

\section*{Introduction}

In this paper,
 we consider metric spaces and their hyperspaces
 endowed with the Hausdorff metric.
Specifically, given a metric space $X = \pair Xd$,
 we shall denote by $\cld(X)$ and $\bcl(X)$ the hyperspaces
 consisting of all nonempty closed sets and
 of all nonempty bounded closed sets in $X$ respectively
 and by $d_H$ the Hausdorff metric,
 which is infinite-valued on $\cld(X)$ if $X$ is unbounded.
We shall sometimes write $\cld_H(X)$ or $\bcl_H(X)$ to emphasize
 the fact that we consider this space with the Hausdorff metric topology.

A theorem of Antosiewicz and Cellina \cite{AnCe} states that,
 given a convex set $X$ in a normed linear space,
 every continuous multivalued map $\map \phi{Y}{\bcl_H(X)}$
 from a closed subset $Y$ of a metric space $Z$,
 can be extended to a continuous map $\map{\ovr f}Z{\bcl_H(X)}$.
Using the language of topology,
 this theorem says that, under the above assumptions,
 $\bcl_H(X)$ is an absolute extensor or an absolute retract
 (in the class of metric spaces).
In \cite{CK}, it is proved that
 the above result is still valid when $X$ is replaced
 by a dense subset of a convex set in a normed linear space.
More generally, $\bcl_H(X)$ is an absolute retract,
 whenever the metric on $X$ is {\em almost convex}
 (see \S\ref{whfijfpiapfi} for the definition).
This condition was further weakened in \cite{KuSaY},
 which has turned out to be actually a necessary and sufficient one
 by Banakh and Voytsitskyy \cite{BaVo}.
In the last paper, several equivalent conditions are given,
 which are too technical to mention them here.
We refer to \cite{BaVo} for the details.

It is a natural question whether $\bcl_H(X)$ and some of its natural subspaces
 are homeomorphic to some standard spaces, like the Hilbert cube/space, etc.
Since the Hausdorff metric topology coincides with the Vietoris topology
 on the hyperspace $\exp(X)$ of nonempty compact sets,
 the above question was already answered, applying known results,
 in case where bounded closed sets in $X$ are compact.
Among the known results,
 let us mention the theorem of Curtis and Schori \cite{CuScho}
 (cf.\ \cite[Chapter 8]{vMill}),
 saying that $\exp(X)$ is homeomorphic to ($\topiso$)
 the Hilbert cube $\Q = [-1,1]^\omega$ if and only if $X$ is a Peano continuum,
 that is, it is compact, connected and locally connected.
Later, Curtis \cite{Curtis} characterized non-compact metric spaces $X$
 for which $\exp(X)$ is homeomorphic to
 the Hilbert cube minus a point $\Qmp$ ($= \Q\setminus\{0\}$)
 or the pseudo-interior $\pseudoint = (-1,1)^\omega$ of $\Q$.\footnote
     {It is well known that $\pseudoint$ is homeomorphic to
     the separable Hilbert space $\ell_2$.}
In particular, $\bcl_H(\Rm) = \exp(\Rm) \topiso \Qmp$.
For more information concerning Vietoris hyperspaces,
 we refer to the book \cite{IlNa}.

The aim of this work is to study topological types
 of some of the natural subspces of the Hausdorff hyperspace.
We consider the following subspaces of $\bcl_H(X)$:
\begin{itemize}
\item
 $\nwd(X)$ --- all nowhere dense closed sets;
\item
 $\perf(X)$ --- all perfect sets;\footnote
     {I.e., completely metrizable closed sets which are dense in itself.}
\item
 $\Cantor(X)$ --- all compact sets homeomorphic to the Cantor set.
\end{itemize}
In case $X = \Rm$ with the standard metric, we can also consider the following subspace:
\begin{itemize}
\item
 $\En(\Rm)$ --- all closed sets of the Lebesgue measure zero.
\end{itemize}
We show that, in case $X=\Rm$,
 the above spaces are homeomorphic to the separable Hilbert space $\ell_2$.
Actually, we prove that if $\Ef$ is one of the above spaces
 then the pair $\pair{\bcl_H(\Rm)}{\Ef}$ is homeomorphic to
 $\pair\Qmp\pseudointmp$.

The completion of a metric space $X = \pair Xd$
 is denoted by $\pair{\til X}d$.
Then $\bcl_H(X,d)$ can be identified with the subspace of $\bcl_H(\til X,d)$,
 via the isometric embedding $A\mapsto \cl_{\til X}A$.
Thus we shall often write $\bcl(X,d)\subs \bcl(\til X,d)$,
 having in mind this identification.
In this case, $\bcl(\til X,d)$ is the completion of $\bcl(X,d)$.
By such a reason, we also consider a dense subspace $D$
 of a metric space $X = \pair Xd$.
For each $0 \loe k < m$, let
$$\nu^m_k
 = \{x = (x_i)_{i=1}^m \in \Rm : x_i \in \R\setminus\Qyu
 \text{ except for at most $k$ many $i$}\},$$
 which is the universal space for completely metrizable subspaces
 in $\Rm$ of $\dim \loe k$.
In case $2k + 1 < m$,
 $\nu^m_k$ is homeomorphic to the $k$-dimensional N{\"o}beling space
 $\nu^{2k+1}_k$,
 which is the universal space for all separable completely metrizable spaces.
Note that $\nu^m_0 = (\R\setminus\Qyu)^m\topiso\R\setminus\Qyu$.
We show that the pairs $\pair{\bcl(\R)}{\bcl(\R\setminus\Qyu)}$
 and $\pair{\bcl(\Rm)}{\bcl(\nu^m_k)}$, $0 \loe k < m-1$,
 are homeomorphic to $\pair\Qmp\pseudointmp$,
 so we have $\bcl_H(\nu^m_k) \topiso \ell_2$
 if $\pair mk = \pair 10$ or $0 \loe k < m-1$.

We also study the space $\cld_H(\R)$. It is very different from the hyperspace $\exp(\R)$. It is not hard to see that $\cld_H(\R)$ has $2^{\aleph_0}$ many components, $\bcl(\R)$ is the only separable one and any other component has weight $2^{\aleph_0}$. We show that a nonseparable component $\Ha$ of $\cld_H(\R)$ is homeomorphic to the Hilbert space $\ell_2(2^{\aleph_0})$ of weight $2^{\aleph_0}$ in case where $\Ha \not\ni \R, [0,\infty), (-\infty,0]$.
This is a partial answer (in case $n = 1$) of Problem 4 in \cite{KuSaY}.

\section{Preliminaries}

We use standard notation concerning sets and topology. For example, we denote by $\omega$ the set of all natural numbers. Given a set $X$, we denote by $[X]^{<\omega}$ the family of all finite subsets of $X$.

Given a metric space $X = \pair Xd$ and a set $A\subs X$, we denote by $\bal(A,r)$ and $\clbal(A,r)$ the open and the closed $r$-balls centered at $A$, that is,
\begin{gather*}
\bal(A,r)=\setof{x\in X}{\dist(x,A)<r} \quad\text{and}\\
 \clbal(A,r)=\setof{x\in X}{\dist(x,A)\loe r}.
\end{gather*}
The Hausdorff metric $d_H$ on $\cld(X)$ is defined as follows:
$$d_H(A,C) = \inf \setof{r>0}{A\subs\bal(C,r)\text{ and }C\subs\bal(A,r)},$$
where $d_H$ is actually a metric on $\bcl(X)$ but $d_H$ is infinite-valued for $\cld(X)$ if $\pair Xd$ is unbounded. The spaces $\cld_H(X)$ and $\bcl_H(X)$ are sometimes denoted by $\cld_H(X,d)$ and $\bcl_H(X,d)$, to emphasize the fact that they are determined by the metric on $X$. In fact, the metric $\rho(x,y) = d(x,y)/(1 + d(x,y))$ induces the same topology on $X$ as $d$ but the Hausdorff metric $\rho_H$ induces a different one on $\cld(X)$. On the other hand, the Hausdorff metric induced by the metric $\bar d(x,y) = \min\{d(x,y), 1\}$ is finite-valued and induces the same topology on $\cld_H(X)$ as $d_H$; moreover $\cld(X)$ is equal to $\bcl(X)$ as sets. Note that the subspace $\Fin(X)=\fin X\setminus\sn\emptyset$ of $\bcl_H(X)$ of all nonempty finite subsets of $X$ is dense in $\bcl_H(X)$ if and only if every bounded set in $X = \pair Xd$ is totally bounded.

\begin{fact}\label{complete-separable}
For a metric space $X = \pair Xd$, the following hold:
\begin{romanenume}
\item
If $d$ is complete then $\pair{\bcl(X,d)}{d_H}$ is a complete metric space
 and the space $\cld_H(X)$ is completely metrizable.
\item
The space $\bcl_H(X,d)$ is separable
 if and only if every bounded set in $X$ is totally bounded.
\end{romanenume}
\end{fact}

We use the standard notation $\exp(X)$ for the Vietoris hyperspace of nonempty compact sets in $X$. Note that $\exp(X)\subs\bcl(X)$ for every metric space $X = \pair Xd$ and it is well known that the Hausdorff metric induces the Vietoris topology on $\exp(X)$. However, if closed bounded sets of $X$ are not compact, then the space $\bcl_H(X)$ is very different from $\bcl_V(X)$ endowed with the Vietoris topology.
We use the following notation:
$$A^-=\setof{C\in \cld(X)}{C\cap A\nnempty} \quad\text{and}\quad A^+=\setof{C\in \cld(X)}{C\subs A},$$
where $A\subs X$. When dealing with $\bcl(X)$ (or other subspace of $\cld(X)$), we still write $A^-$ and $A^+$ instead of $A^-\cap \bcl(X)$ and $A^+\cap \bcl(X)$ respectively.

In the rest of this section, we shall give preliminary results of infinite-dimensional topology. For the details, we refer to the book \cite{BRZ}. We abbreviate ``absolute neighborhood retract'' to ``ANR''.

Let $X = \pair Xd$ be a metric space. It is said that a map $\map fYX$ can be {\em approximated} by maps in a class $\Ef$ of maps if for every map $\map \alpha X{(0,1)}$ there exists a map $\map gYX$ which belongs to $\Ef$ and such that $d(f(y),g(y))<\alpha(f(y))$ for every $y\in Y$. A closed subset $A \subs X$ is a {\em \Z-set} in $X$ if the identity map $\id_X$ of $X$ can be approximated by maps $\map fXX$ such that $\img fX\cap A=\emptyset$. Strengthening the last condition to $\cl_X(\img fX)\cap A=\emptyset$, we define the notion of a {\em strong \Z-set}. In case $X$ is locally compact, every \Z-set in $X$ is a strong \Z-set. Moreover, it is well known that every \Z-set in an $\ell_2$-manifold is a strong \Z-set. A countable union of (strong) \Z-sets is called a ({\em strong}) {\em \Zsig-set}. We call $X$ a ({\em strong}) {\em \Zsig-space} if it is a (strong) \Zsig-set in itself. An embedding $\map fXY$ is called a {\em \Z-embedding} if $\img fX$ is a \Z-set in $Y$.

It is said that $D\subs X$ is {\em homotopy dense} in $X$ if there exists a homotopy $\map h{X\times[0,1]}X$ such that $h_0 = \id$ and $\img{h_t}X\subs D$ for every $t > 0$, where $h_t(x)=h(x,t)$. The complement of a homotopy dense subset of $X$ is said to be {\em homotopy negligible}. If $A\subs X$ is a homotopy negligible closed set then $A$ is a \Z-set in $X$.

\begin{fact}\label{Z-set}
For a closed set $A$ in an ANR $X$,
 the following are equivalent:
\begin{alphenume}
\item
 $A$ is a \Z-set in $X$;
\item
 each map $\map f{[0,1]^n}X$, $n\in\nat$,
 can be approximated by maps into $X\setminus A$;
\item
 $A$ is homotopy negligible in $X$.
\end{alphenume}
\end{fact}

\begin{fact}\label{sedgfasf}
Let $D$ be a homotopy dense subset of an ANR $X$. Then the following hold:
\begin{romanenume}
\item
 $D$ is also an ANR.
\item
 A closed set $A\subs X$ is a \Z-set in $X$ if and only if $A\cap D$ is a \Z-set in $D$.
\item
 If $A\subs X$ is a strong \Z-set in $X$ then $A\cap D$ is a strong \Z-set in $D$.
\end{romanenume}
\end{fact}

\begin{prop}\label{wetafqwtrqf}
Assume that $X$ is a homotopy dense subset of a \Q-manifold $M$. Then $X$ is an ANR and every \Z-set in $X$ is a strong \Z-set. Furthermore, $X$ is a strong \Zsig-space if and only if $X$ is contained in a \Zsig-set in $M$.
\end{prop}

\begin{proof}
We verify only the ``furthermore" statement. Assume $X\subs\bigcup_{\ntr}Z_n$, where each $Z_n$ is a \Z-set in $M$. Then each $Z_n$ is a strong \Z-set in $M$, because $M$ is locally compact, and therefore by Fact \ref{sedgfasf} (iii), each $Z_n\cap X$ is a strong \Z-set in $X$.
Conversely, if $X=\bigcup_{\ntr}X_n$, where each $X_n$ is a (strong) \Z-set in $X$, then by Fact \ref{sedgfasf} (ii), $Z_n=\cl_{M}X_n$ is a \Z-set in $M$. Clearly, $X\subs\bigcup_{\ntr}Z_n$.
\end{proof}

Let $\Cee$ be a topological class of spaces,
 that is, if $X$ is homeomorphic to some $Y \in \Cee$ then $X$ also belongs to $\Cee$. It is said that $\Cee$ is {\em open} (resp.\ {\em closed}) {\rm hereditary} if $X\in\Cee$ whenever $X$ is an open (resp.\ closed) subspace of some $Y\in\Cee$. A space $X$ is called {\em strongly $\Cee$-universal} if for every $Y\in\Cee$ and every closed subset $A \subs Y$, every map $\map fYX$ such that $f\rest A$ is a \Z-embedding can be approximated by \Z-embeddings $\map gXY$ such that $g\rest A=f\rest A$. Similarly, one defines {\em $\Cee$-universality}, relaxing the above condition to the case $A=\emptyset$, that is, $X$ is {\em $\Cee$-universal} if every map $\map fYX$ of $Y\in\Cee$ can be approximated by \Z-embeddings.

\begin{fact}\label{univ-strong}
Let $X$ be an ANR such that every \Z-set in $X$ is strong and let $\Cee$ be an open and closed hereditary topological class of spaces. If every open subspace $U\subs X$ is $\Cee$-universal then $X$ is strongly $\Cee$-universal.
\end{fact}

Given a topological class $\Cee$ of spaces, we denote by $\sigma\Cee$ the class of all spaces of the form $X=\bigcup_{\ntr}X_n$, where each $X_n$ is closed in $X$ and $X_n \in\Cee$. Recall that $X$ is a {\em $\Cee$-absorbing space} if $X \in \sigma\Cee$ is a strongly $\Cee$-universal ANR which is a strong \Zsig-space. In case $\Cee$ is closed hereditary, we can write $X=\bigcup_{\ntr}X_n$, where each $X_n$ is a strong \Z-set in $X$ and $X_n \in \Cee$.

We shall denote by $\Emm_0$ and $\Emm_1$ the classes of all compact metrizable spaces and all Polish spaces\footnote{I.e., separable completely metrizable spaces.} respectively. Let $\pseudobd=\Q\setminus\pseudoint$ denote the pseudo-boundary\footnote{In some articles (e.g. \cite{BRZ}), $\Sigma$ denotes the {\em radial interior} of \Q, i.e., $\Sigma=\setof{x\in \Q}{\sup_{\ntr}|x(n)|<1}$. However, there is an auto-homeomorphism of $\Q$ which maps the pseudoboundary onto the radial interior.} of $\Q$.

\begin{fact}\label{3e4gwdfs}
If $X$ is an $\Emm_0$-absorbing homotopy dense subspace of\/ \Q,
 then $\pair \Q X\topiso \pair \Q\pseudobd$.
In case $X\subs\Qmp$, $\pair{\Qmp}{X}\topiso \pair{\Qmp}{\pseudobd}$.
\end{fact}

\begin{fact}\label{weosaijfpajf}
Assume that $X$ is a both homotopy dense and homotopy negligible subset of a Hilbert cube manifold $M$. If $X$ is $\sigma$-compact then it is a strong \Zsig-space.
\end{fact}

\begin{proof}
Assume $X=\bigcup_{\ntr}K_n$, where each $K_n$ is compact. Then each $K_n$ is closed in $M$ and therefore it is a strong \Z-set by Fact \ref{sedgfasf} (iii).
\end{proof}

\section{Borel classes of several Hausdorff hyperspaces}\label{borelclasses}

Let $\pair{\til X}d$ denote the completion of $\pair Xd$.
We identify $\bcl(X,d)$ with the subspace of $\bcl(\til X,d)$,
 via the isometric embedding $A\mapsto \cl_{\til X}A$.
Then, $\pair{\bcl(\til X)}{d_H}$ is a completion of $\pair{\bcl(X)}{d_H}$.
Moreover, it should be noticed that
 $A\in \bcl(\til X)\setminus \bcl(X)$ if and only if
 $A\ne \cl_{\til X}(A\cap X)$.
Saint Raymond proved in \cite[Th\'eor\`eme 1]{S} that
 if $X$ is the union of a Polish subset and a $\sigma$-compact subset
 then $\bcl_H(X)$ is $F_{\sigma\delta}$
 (hence Borel) in $\bcl_H(\til X)$.\footnote
	{In \cite{S}, $X$ is assumed to be a subspace of a compact metric space,
	but the proof is valid without this assumption.
	Moreover, it is also proved in \cite[Th\'eor\`eme 6]{S} that
	if $\bcl_H(X)$ is absolutely Borel (i.e., Borel in its completion)
	then $X$ is the union of a Polish subset and a $\sigma$-compact subset.}
In particular, we have the following:

\begin{prop}\label{rthsdpio}
If $X = \pair Xd$ is $\sigma$-compact
 then the space $\pair{\bcl(X)}{d_H}$ is $F_{\sigma\delta}$
 in its completion $\pair{\bcl(\til X)}{d_H}$.
\end{prop}

Moreover, the following can be easily obtained
 by adjusting the proof of \cite[Th\'eor\`eme 1]{S}:\footnote
     {A similar result was proved by Costantini \cite{Cost}
     for the Wijsman topology.}

\begin{prop}\label{owteepgsdgfa}
If $X = \pair Xd$ is Polish ($d$ is not necessarily complete) then the space $\pair{\bcl(X)}{d_H}$ is $G_\delta$ in its completion $\pair{\bcl(\til X)}{d_H}$.
\end{prop}

For the readers' convenience, direct short proofs of the above two propositions
are given 
in the Appendix.
Combining Fact \ref{complete-separable} and Proposition \ref{owteepgsdgfa},
 we have the following:

\begin{wn}\label{ppojnkjiu}
If $X = \pair Xd$ is Polish in which every bounded set is totally bounded, then the space $\bcl_H(X)$ is also Polish.
\end{wn}

Concerning the spaces $\nwd(X)$ and $\perf(X)$,
 we prove here the following:

\begin{prop}\label{wetgivwet}
For every separable metric space $X$, the space $\nwd(X)$ is $G_\delta$ in $\bcl_H(X)$.
\end{prop}

\begin{proof}
Let $\ciag U$ be a countable open base for $X$. For each $\ntr$, let
$$\Ef_n=\setof{A\in\bcl(X)}{U_n\subs A}.$$
Then each $\Ef_n$ is closed in $\bcl_H(X)$ and $\bigcup_{\ntr}\Ef_n=\bcl(X)\setminus\nwd(X)$.
\end{proof}

\begin{prop}\label{oeihgwef}
If $X$ is locally compact then $\perf(X)$ is $G_\delta$ in $\bcl_H(X)$.
\end{prop}

\begin{proof}
Let $\ciag U$ enumerate an open base of $X$ such that $\cl U_n$ is compact for every $\ntr$.
Note that, by compactness, $(\cl U_n)^-$ is closed in $\bcl_H(X,d)$.
For each $n,m\in\nat$ define
$$\Phi(n,m)=\setof{\pair kl\in\nat^2}{U_k\cap U_l=\emptyset,\ U_k\cup U_l\subs \bal(U_n,1/m)}.$$
We claim that
$$\bcl(X,d)\setminus\perf(X)=\bigcup_{n,m\in\nat}\bigcap_{\pair kl\in\Phi(n,m)} \Bigl( (\cl U_n)^-\setminus(U_k^-\cap U_l^-)\Bigr).$$
The set on the right-hand side is $F_{\sigma}$, so this will finish the proof.

Note that a closed set in a Polish space is perfect if and only if it has no isolated points.
If $A\in \bcl(X,d)\setminus\perf(X)$ then there is $y\in A$ which is isolated in $A$. We can find $n,m\in\nat$ such that $y\in U_n$ and $\bal(U_n,1/m)\cap A=\sn y$. Then $A\in (\cl U_n)^-$ and $A\notin U_k^-\cap U_l^-$ whenever $\pair kl\in\Phi(n,m)$.

Conversely, assume that there are $n,m\in\nat$ such that $A\in(\cl U_n)^-$ and $A\notin U_k^-\cap U_l^-$ for every $\pair kl\in\Phi(n,m)$. Then $A\cap \bal(U_n,1/m)\nnempty$ and the second condition says that $A\cap \bal(U_n,1/m)$ does not contain two points, so it is a singleton. Thus $A\notin\perf(X)$.
\end{proof}

Replacing $(\cl U_n)^-$ by $\bal(U_n,1/m)^-$ in the formula from the proof above, we obtain the following:

\begin{wn}
The space $\perf(X)$ is $F_{\sigma\delta}$ in\/ $\bcl_H(X)$ if $X$ is Polish.
\end{wn}

Since $\Cantor(\R^m) = \perf(\R^m) \cap \nwd(\R^m)$, the following is a combination of Propositions \ref{wetgivwet} and \ref{oeihgwef}:

\begin{wn}
The space $\Cantor(\R^m)$ is $G_\delta$ in\/ $\bcl_H(\R^m)$.
\end{wn}

Now, we shall prove the following:

\begin{prop}
The space $\En(\R^m)$ is Polish.
\end{prop}

\begin{proof}
Let $\ciag I$ enumerate all open rational cubes (i.e. products of rational intervals) in $\Rm$. Given $k\in\nat$, we define
$$S_k= \Bigl\{s\in\fin\nat : \sum_{n\in s}|I_n|<1/k\Bigr\},$$
where $|I|$ denotes the volume of the cube $I\subs \Rm$.
We claim that
$$\En(\R^m)=\bigcap_{k\in\nat}\bigcup_{s\in S_k}\Bigl(\bigcup_{n\in s}I_n\Bigr)^+.$$
Clearly, if $A$ belongs to the right-hand side then for each $k\in\nat$ there is $s\subs\nat$ such that $A\subs\bigcup_{n\in s}I_n$ and $\sum_{n\in s}|I_n|<1/k$; therefore $A$ has Lebesgue measure zero.

Assume now $A$ has Lebesgue measure zero and fix $k<\nat$. Then $A\subs\bigcup_{\ntr}J_n$, where each $J_n$ is an open rational cube and $\sum_{\ntr}|J_n|<1/k$. By compactness, $A\subs J_0\cup\dots\cup J_{l-1}$ for some $m$ and $\{J_0,\dots,J_{l-1}\}=\setof{I_n}{n\in s}$ for some $s\in S_k$. Thus $A\in\bigcup_{s\in S_k}(\bigcup_{n\in s}I_n)^+$.
\end{proof}

\section{Almost convex metric spaces}\label{whfijfpiapfi}

Recall that a metric $d$ on $X$ is {\em almost convex} if for every $\alpha>0$, $\beta>0$ and for every $x,y\in X$ such that $d(x,y)<\alpha+\beta$, there exists $z\in X$ with $d(x,z)<\alpha$ and $d(z,y)<\beta$.

Fix a dense set $X$ in a separable Banach space $E$. Let $d$ denote the metric on $X$ induced by the norm of $E$. Then $\pair Xd$ is an almost convex metric space and therefore by a result of \cite{CK} the space $\bcl(X,d)$ is an absolute retract. In case where $X$ is $G_\delta$, the space $\bcl(X,d)$ is completely metrizable by Proposition \ref{owteepgsdgfa}. If additionally $E$ is finite-dimensional then $\bcl(X,d)$ is Polish by Corollary \ref{ppojnkjiu}. In case where $X$ is $\sigma$-compact, by Proposition \ref{rthsdpio}, $\bcl(X,d)$ is absolutely $F_{\sigma\delta}$. It is natural to ask whether these spaces or their subspaces, discussed in \S\ref{borelclasses}, are homeomorphic to some standard spaces. Such standard spaces appear as homotopy dense subspaces of the Hilbert cube \Q.

Let $\uclb(X,d)$ denote the family of all sets of the form $\clbal(C,t)$, the closed $t$-neighborhood of $C\in\bcl(X,d)$, where $t>0$.

\begin{prop}\label{sdgeriphgpwo}
If $\pair Xd$ is an almost convex metric space then the subspace $\uclb(X,d)$ is homotopy dense in $\bcl(X,d)$.
\end{prop}

\begin{proof}
Define a homotopy $\map h{\bcl(X,d)\times[0,1]}{\bcl(X,d)}$ by the formula:
$$h(A,t)=\clbal(A,t).$$
It suffices to verify the continuity of $h$ with respect to Hausdorff metric topology. It has been checked in \cite{CK} that $d_H(\clbal(A,t),\clbal(A,s))\loe|t-s|$. Thus we have
\begin{align*}
d_H(h(A,t),h(B,s))&\loe d_H(h(A,t),h(A,s))+d_H(h(A,s),h(B,s))\\
&\loe |t-s|+d_H(h(A,s),h(B,s)).
\end{align*}
It remains to check that $d_H(\clbal(A,s),\clbal(B,s))\loe d_H(A,B)$.

To complete the proof, we show the following:
$$r>d_H(A,B),\ \eps>0 \Longrightarrow
 r+\eps\goe d_H(\clbal(A,s),\clbal(B,s)),$$
For this aim, it suffices to check that $\clbal(A,s)\subs \bal(\clbal(B,s),r+\eps)$; then by symmetry we shall also get $\clbal(B,s)\subs \bal(\clbal(A,s),r+\eps)$.

For each $x\in \clbal(A,s)$, choose $a\in A$ such that $d(x,a)<s+\eps$. There is $b\in B$ such that $d(a,b)<r$. Then we have $d(x,b)<s+r+\eps$. Using the almost convexity of $d$, we can find $y$ such that $d(b,y)<s$ and $d(y,x)<r+\eps$. Then $y\in\bal(B,s)$ and hence $x\in\bal(y,r+\eps)\subs\bal(\clbal(B,s),r+\eps)$.
\end{proof}

Denote by $\reg(X,d)$ the hyperspace of all nonempty bounded regularly closed subsets of a metric space $\pair Xd$.
Clearly, $\uclb(X,d)\subs\reg(X,d)$.

\begin{wn}\label{owehfafafs}
Let $\pair Xd$ be an almost convex metric space and $D\subs X$ a dense set. Then the spaces $\reg(X,d)$ and $\bcl(D,d)$ are homotopy dense in $\bcl(X,d)$.
\end{wn}

\begin{proof}
Regarding $\bcl(D,d)\subs\bcl(X,d)$ via the embedding $A\mapsto \cl_X A$, we have $\reg(X,d)\subs\bcl(D,d)$. This follows from the fact that $\cl(D\cap U)=\cl U$ for every open set $U\subs X$. Since $\uclb(X,d)$ is homotopy dense in $\bcl(X,d)$ by Proposition \ref{sdgeriphgpwo} and $\uclb(X,d)\subs\reg(X,d)$, we have the result.
\end{proof}

\section{Strict deformations}

Assume we are looking at certain homotopy dense subspaces of the Hilbert cube \Q. Let $X \sups X_0$ be such spaces. If $X_0 \topiso \pseudobd$ then, in order to conclude that $\pair\Q X\topiso \pair\Q\pseudobd$, it suffices to check that $X$ is a \Zsig-set in \Q, by applying \cite[Theorem 6.6]{Chapman}.
However, to see that $X_0 \topiso \pseudobd$, we have to check that $X_0$ is strongly $\Emm_0$-universal. Below is a tool which simplifies this step. To formulate it, we need some extra notions concerning homotopies.

A homotopy $\map\phi{X\times [0,1]}X$ is called a {\em strict deformation} if $\phi_0=\id$ and
$$\phi(x,t)=\phi(x',t')\land t>0\land t'>0\implies x=x'.$$
It is said that $\phi$ {\em omits} $A\subs X$ if $\img\phi{X\times(0,1]}\cap A=\emptyset$. Finally, we say that a space $X$ is {\em strictly homotopy dense} in $Y$ if $X\subs Y$ and there exists a strict deformation which omits $Y\setminus X$ (so in particular $X$ is homotopy dense in $Y$).

\begin{lm}\label{mbysld}
For every $\Z$-set $A$ in a $\Q$-manifold $M$, there exists a strict deformation of $M$ which omits $A$.
\end{lm}

\begin{proof}
Find a \Z-embedding $\map{f_0}MM$
 which is properly $2^{-2}$-homotopic to the identity
 and so that $\img{f_0}M\cap A=\emptyset$.
Further, find a \Z-embedding $\map{f_1}MM$
 which is properly $2^{-3}$-homotopic to the identity
 and $\img{f_1}M\cap(\img{f_0}M\cup A)=\emptyset$.
Continuing this way,
 we find \Z-embeddings $\map{f_n}MM$, $\ntr$,
 such that $f_n$ is properly $2^{-n-2}$-homotopic to the identity and
$$\img{f_n}M\cap (\img{f_{n-1}}M\cup \dots\cup \img{f_0}M\cup A)=\emptyset.$$
Then, we have proper $2^{-(n+1)}$-homotopies $\map{g^n}{M\times[0,1]}M$,
 $\ntr$, such that $g^n_0 = f_n$ and $g^n_1 = f_{n+1}$.
We can define a homotopy $\map{g}{M\times[0,1]}M$ by
 $g(x,0) = x$ and
$$g(x,t) = g^n(x,2 - 2^{n+1}t)
 \;\text{ for $2^{-(n+1)} \leqslant t \leqslant 2^{-n}$, $\ntr$.}$$
Note that $g_{2^{-n}} = f_n$ for each $\ntr$,
 each $g\rest M\times[2^{-n-1},2^{-n}]$ is proper
 and $2^{-n-1}$-close to the projection $\map{\pr_M}{M\times(0,2^{-n}]}M$.
The continuity of $g$ at $(x,0)$ is guaranteed by the last fact.
Using the strong $\Emm_0$-universality of $M$
 (see \cite[Theorem 1.1.26]{BRZ}),
 we can inductively obtain $\map{h_n}{M\times[0,1]}M$, $\ntr$, such that
\begin{enumerate}
\item
 $h_n\rest M\times[2^{-n-1},1]$ is a $Z$-embedding,
\item
 $h_n\rest M\times[2^{-n},1] = h_{n-1}\rest M\times[2^{-n},1]$,
\item
 $h_n\rest M\times[0,2^{-n-1}] = g\rest M\times[0,2^{-n-1}]$,
\item
 $h_n\rest M\times[2^{-n-1},2^{-n}]$ is $2^{-n-1}$-close
 to $g\rest M\times[2^{-n-1},2^{-n}]$,
 hence it is $2^{-n}$-close to $\map{\pr_M}{M\times[2^{-n-1},2^{-n}]}M$,
\item
 $\img{h_n}{M\times[2^{-n-1},1]}$ is disjoint from $A$.
\end{enumerate}
Finally, the limit $h = \lim_{n\to\infty} h_n$ is the desired one.
\end{proof}

\begin{tw}\label{ejgrpio}
Assume that $X$ is a \Zsig-subset of a $\Q$-manifold $M$ which is strictly homotopy dense in $M$. Then $X$ is an $\Emm_0$-absorbing space. In particular, if $M\topiso\Q$ then $\pair MX\topiso\pair\Q\pseudobd$ and if $M\topiso\Qmp$ then $\pair MX\topiso \pair{\Qmp}{\pseudobd}$.
\end{tw}

\begin{proof}
The assumption says in particular that $X$ is homotopy dense in $M$, so it follows from Proposition \ref{wetafqwtrqf} that $X$ is an ANR being a strong \Zsig-space. It remains to check that $X$ is strongly $\Emm_0$-universal. For the additional statement, we can just apply Fact \ref{3e4gwdfs}.

Fix a map $\map fAX$ of a compact metric space such that $f\rest B$ is a \Z-embedding, where $B\subs A$ is closed. Note that every compact subset of $X$ is a \Z-set in $M$, hence it is a \Z-set in $X$ by Fact \ref{sedgfasf} (ii), so we just have to preserve $f\rest B$, not worrying about \Z-sets. We assume that $A$ is endowed with the metric such that $\diam(A)\loe1$. Fix $\eps>0$. Using the strong $\Emm_0$-universality of $M$ (see \cite[Theorem 1.1.26]{BRZ}), we can find a \Z-embedding $\map gAM$ which is $\eps/2$-close to $f$ and such that $\img g{A\setminus B}\cap X=\emptyset$ (here we use the fact that $X$ is a $\Z_{\sigma}$-set in $M$~and also that $\img fB$ is a \Z-set in $M$).

By Lemma \ref{mbysld}, we have a strict deformation $\map \phi{M\times[0,1]}M$ which omits $\img fB$. Fix a metric $d$ for $M$ and choose a map $\map\gamma A{[0,1]}$ so that $\gamma^{-1}(0)=B$ and
$$d(g(a),\phi(g(a),\gamma(a)))<\eps/4 \;\text{ for every $a\in A$.}$$
On the other hand, by the assumption, there is a strict deformation $\map \psi{M\times[0,1]}M$ which omits $M\setminus X$. Define $\map hAX$ by setting
$$h(a)=\psi(\phi(g(a),\gamma(a)),\delta(a)),$$
where $\map\delta A{[0,1]}$ is a map chosen so that $B=\delta^{-1}(0)$ and
$$d(h(a),\phi(g(a),\gamma(a)))<\min\{\eps/4,\ \dist(\phi(g(a),\gamma(a)),\img fB)\}.$$
This ensures us that $h$ is $\eps/2$-close to $g$ and that $h(a)\notin \img fB$ whenever $a\in A\setminus B$.
Then $h$ is a map which is $\eps$-close to $f$ and $\img hA\subs X$. Furthermore, $h\rest B=g\rest B=f\rest B$. It remains to check that $h$ is one-to-one (then it is a \Z-embedding, since every compact set in $X$ is a \Z-set).

Suppose $h(a)=h(a')$. If $a,a'\in B$ then $g(a)=g(a')$ and consequently $a=a'$. When $a,a'\in A\setminus B$, since $\psi$ and $\phi$ are strict deformations, $g(a)=g(a')$ and hence $a=a'$. In case $a\in B$ and $a'\notin B$, we have $h(a)=g(a)=f(a)\in\img fB$ but $h(a')\notin \img fB$ because $\phi$ omits $\img fB$. Thus, this case does not occur.
\end{proof}

\section{Pseudo-interiors of $\bcl_H(\Rm)$}

Throughout this section,
 $m>0$ is a fixed natural number.
A particular case of a well known theorem of Curtis \cite{Curtis} says
 that $\bcl_H(\Rm)=\exp(\Rm)$ is homeomorphic to $\Qmp$.
We shall consider the standard (convex) Euclidean metric $d$ on $\Rm$.
In this section, we investigate various $G_\delta$ subspaces of $\bcl_H(\Rm)$.
The main result of this section is the following:

\begin{tw}\label{pint-hyperspace}
Let $\Ef \subs \bcl_H(\Rm)$ be one of the subspaces below:
$$\nwd(\Rm),\ \perf(\Rm),\ \Cantor(\Rm),\ \En(\Rm),\ \bcl(D),$$
where $D$ is a dense $G_\delta$ set in $\Rm$ such that\/
 $\Rm \setminus D$ is also dense in $\Rm$ and
 in case $m > 1$ it is assumed that $D = \img pD \times \R$,
 where $p : \Rm \to \R^{m-1}$ is the projection onto the first $m-1$ coordinates.
Then the pair\/ $\pair{\bcl(\Rm)}{\Ef}$ is homeomorphic to
 $\pair\Qmp\pseudointmp$.
\end{tw}

Applying Theorem \ref{pint-hyperspace} above, we have

\begin{wn}
Suppose $\pair mk = \pair 10$ or $0 \loe k < m - 1$.
Then,
$$\pair{\bcl(\Rm)}{\bcl(\nu^m_k)} \topiso \pair\Qmp\pseudointmp.$$
Consequently, $\bcl_H(\nu^m_k) \topiso \ell_2$.
\end{wn}

\begin{proof}
As a direct consequence of Theorem \ref{pint-hyperspace}, we have
$$\pair{\bcl(\R)}{\bcl(\nu^1_0)} = \pair{\bcl(\R)}{\bcl(\R\setminus\Qyu)}
 \topiso \pair\Qmp\pseudointmp.$$
For each $0 \loe k < m - 1$,
 observe that
 $\Rm \setminus (\nu^{m-1}_k\times\R)
 = (\R^{m-1} \setminus \nu^{m-1}_k)\times\R
 \subs \Rm \setminus \nu^m_k$.
Thus, it follows that
$$\bcl(\Rm) \setminus \bcl(\nu^{m-1}_k\times\R)
 \subs \bcl(\Rm) \setminus \bcl(\nu^m_k).$$
By Proposition \ref{owteepgsdgfa} and Corollary \ref{owehfafafs},
 $\bcl(\nu^m_k)$ is a homotopy dense $G_\delta$ set in $\bcl_H(\Rm)$,
 which implies that $\bcl(\Rm) \setminus \bcl(\nu^m_k)$
 is a \Zsig-set in $\bcl(\Rm)$.
On the other hand,
 we can apply Theorem \ref{pint-hyperspace} to obtain
$$\pair{\bcl(\Rm)}{\bcl(\Rm)\setminus\bcl(\nu^{m-1}_k\times\R)}
 \topiso \pair{\Qmp}{\pseudobd}.$$
Then, it follows from Theorem 6.6 in \cite{Chapman} that
$$\pair{\bcl(\Rm)}{\bcl(\Rm)\setminus\bcl(\nu^m_k)}
 \topiso \pair\Qmp\pseudobd.$$
Thus, we have the result.
\end{proof}

The conclusion of Theorem \ref{pint-hyperspace} is equivalent to
$$\pair{\bcl_H(\Rm)}{\bcl_H(\Rm)\setminus\Ef} \topiso \pair\Qmp\pseudobd.$$
We saw in \S\ref{borelclasses} that
 the subspace $\Ef \subs \bcl(\Rm)$ in Theorem \ref{pint-hyperspace}
 is $G_\delta$, that is,
 $\bcl_H(\Rm)\setminus\Ef$ is $F_\sigma$ in $\bcl_H(\Rm)$.
If $\Ef$ contains a homotopy dense subset of $\bcl_H(\Rm)$
 then the complement $\bcl_H(\Rm)\setminus\Ef$ is a \Zsig-set.
Thus, in order to apply Theorem \ref{ejgrpio} to obtain the result,
 it suffices to show that
 $\Ef$ contains a homotopy dense subset of $\bcl_H(\Rm)$
 and the complement $\bcl_H(\Rm)\setminus\Ef$ contains
 a strictly homotopy dense subset of $\bcl_H(\Rm)$.
Observe that
$$\Fin(\Rm) \subs \En(\Rm) \subs \nwd(\Rm) \quad\text{and}\quad
 \Cantor(\Rm)\subs\perf(\Rm).$$

As a special case of a well known result
 due to Curtis and Nguyen To Nhu \cite{CuNhu},
 we have
$$\pair{\bcl_H(\Rm)}{\Fin(\Rm)} = \pair{\exp(\Rm)}{\Fin(\Rm)}
 \topiso\pair\Qmp{\Q_f\setminus0},$$
where $\Q_f$ denotes the subspace of $\Q$
 consisting of all eventually zero sequences,
 which is homotopy dense in $\Q$.
This fact implies the following:

\begin{lm}\label{fin-h-dense}
The subspace $\Fin(\Rm)$ is homotopy dense in $\bcl_H(\Rm)$.
\end{lm}

Using Lemma \ref{fin-h-dense} above, we can easily show the following:

\begin{lm}
The space $\Cantor(\Rm)$ is homotopy dense in $\bcl_H(\Rm)$.
\end{lm}

\begin{proof}
Let $h$ be a homotopy of $\bcl_H(\Rm)$ which witnesses
 that $\Fin(\Rm)$ is homotopy dense,
 i.e., $h(A,t)$ is a finite set for every $t > 0$.
Choose a Cantor set $C \subs [0,1]^m$ with $0 \in C$
 and define a homotopy $\map\phi{\bcl_H(\Rm)\times[0,1]}{\bcl_H(\Rm)}$ by
$$\phi(A,t) = h(A,t) + tC.$$
Then $\phi_0 = \id$ and $\phi(A,t)\in\Cantor(\Rm)$ for every $t>0$
 because a finite union of Cantor sets is a Cantor set.
\end{proof}

Concerning the space $\bcl(D)$ in Theorem \ref{pint-hyperspace},
 we have shown in Corollary \ref{owehfafafs} that
 it is homotopy dense in $\bcl_H(\Rm)$.
Thus, to complete the proof of Theorem \ref{pint-hyperspace},
 it remains to show the following:

\begin{lm}\label{strict-homotopy}
Under the same assumption as Theorem \ref{pint-hyperspace},
 each of the following spaces
 are strictly homotopy dense in $\bcl_H(\Rm)$:
$$\bcl(\Rm) \setminus \nwd(\Rm),\ \bcl(\Rm) \setminus \perf(\Rm),\
 \bcl(\Rm) \setminus \bcl(D).$$
\end{lm}

First, we show the following lemma,
 which also gives a direct proof of Lemma \ref{fin-h-dense}:

\begin{lm}\label{sdegfaqqfas}
For $D \subs \Rm$,
 if $\Rm \setminus D$ is dense in $\Rm$
 then $\Fin(\Rm) \setminus \bcl(D)$ is homotopy dense in $\bcl_H(\Rm)$.
\end{lm}

\begin{proof}
Let $\Ha = \Fin(\Rm) \setminus \bcl(D)$, that is,
 $\Ha$ consists of all nonempty finite sets $A \subs \Rm$
 such that $A \setminus D \nnempty$.
Then $\Ha$ is dense in $\bcl_H(\Rm)$.
Moreover, $\Ha$ is closed under finite unions,
 i.e., $A \cup B \in \Ha$ whenever $A, B \in\Ha$.
Recall that $\pair{\bcl_H(\Rm)}\cup$ is a Lawson semilattice
 (see \cite{Lawson}), that is,
 the union operator $\pair AB \mapsto A \cup B$ is continuous
 and $\bcl_H(\Rm)$ has an open base consisting of subsemilattices;
 namely, every open ball with respect to the Hausdorff metric
 is a subsemilattice of $\pair{\bcl_H(\Rm)}\cup$.
By virtue of \cite[Theorem 5.1]{KSY},
 it suffices to show that $\Ha$ is relatively $LC^0$ in $\bcl_H(\Rm)$.
Recall that a subspace $Y$ of a space $X$ is {\em relatively $LC^0$ in} $X$
 if every neighborhood $U$ of each $x \in X$ contains a neighborhood
 $V$ of $x$ in $X$ such that
 every $a,b \in V \cap Y$ can be joined by a path in $U \cap Y$.

Fix $A \in \bcl_H(\Rm)$ and $\eps > 0$.
For each $A_0,A_1 \in \bal_{d_H}(A,\eps/2) \cap \Ha$,
 we describe how to construct a path in $\bal_{d_H}(A,\eps)\cap \Ha$
 which joins $A_0$ to $A_0 \cup A_1$.
Let $A_1 = \{p_0,\dots,p_{n-1}\}$.
For each $i < n$, find $q_i\in A_0$ such that $\norm{p_i-q_i} < \eps/2$,
 and define
$$h(t) = A_0 \cup \setof{(1-t)q_i + tp_i}{i<n}
 \quad\text{for each $t \in [0,1]$.}$$
Then $h(t) \in \Ha$ because $A_0 \subs h(t) \in \Fin(\Rm)$.
Further, $d_H(A_0,h(t)) < \eps/2$, that is, $h(t)\in \bal_{d_H}(A,\eps)$.
Finally, $h(0) = A_0$ and $h(1) = A_0\cup A_1$.
By the same argument,
 we can construct a path in $\bal_{d_H}(A,\eps) \cap \Ha$
 which joins $A_0 \cup A_1$ to $A_1$.
\end{proof}

\begin{proof}[Proof of Lemma \ref{strict-homotopy}]
First, we show the case $m=1$.
It suffices to construct a strict deformation
 $\map\phi{\bcl_H(\R)\times[0,1]}{\bcl_H(\R)}$
 which omits $\nwd(\R) \cup \perf(\R) \cup \bcl(D)$.
Let $h$ be a homotopy of $\bcl(\R)$ which witnesses
 that $\Fin(\R)\setminus\bcl(D)$ is homotopy dense (Lemma \ref{sdegfaqqfas}).
Since $\bcl_H([1,2]) \topiso \Q$,
 we have an embedding $g : \bcl_H(\R) \to \bcl_H([1,2])$.
The desired $\phi$ can be defined as follows:
$$\phi(A,t) = h(A,t) \cup \{\max h(A,t) + [t,2t],\
 \min h(A,t) - tg(A)\}.$$
\begin{figure}[hbtp]
\begin{center}\small
\unitlength 0.1in
\begin{picture}(42.00,8.60)(4.00,-11.40)
\put(28.0000,-8.0000){\makebox(0,0){$*$}}%
\put(22.0000,-8.0000){\makebox(0,0){$*$}}%
\put(20.0000,-8.0000){\makebox(0,0){$*$}}%
\put(25.0000,-8.0000){\makebox(0,0){$*$}}%
\put(25.1000,-9.7000){\makebox(0,0)[lt]{$h(A,t)$}}%
\put(30.0000,-8.0000){\makebox(0,0){$*$}}%
%
\special{pn 20}%
\special{pa 3600 800}%
\special{pa 4200 800}%
\special{fp}%
%
\special{pn 4}%
\special{pa 1900 870}%
\special{pa 1900 920}%
\special{fp}%
%
\special{pn 4}%
\special{pa 1900 920}%
\special{pa 3000 920}%
\special{fp}%
\special{pa 3000 920}%
\special{pa 3000 870}%
\special{fp}%
\put(19.0000,-8.0000){\makebox(0,0){$*$}}%
%
\special{pn 4}%
\special{pa 3000 500}%
\special{pa 3000 750}%
\special{fp}%
\special{sh 1}%
\special{pa 3000 750}%
\special{pa 3020 683}%
\special{pa 3000 697}%
\special{pa 2980 683}%
\special{pa 3000 750}%
\special{fp}%
\put(28.8000,-4.5000){\makebox(0,0)[lb]{$\max h(A,t)$}}%
%
\special{pn 4}%
\special{pa 1900 500}%
\special{pa 1900 750}%
\special{fp}%
\special{sh 1}%
\special{pa 1900 750}%
\special{pa 1920 683}%
\special{pa 1900 697}%
\special{pa 1880 683}%
\special{pa 1900 750}%
\special{fp}%
\put(18.0000,-4.5000){\makebox(0,0)[lb]{$\min h(A,t)$}}%
\put(34.4000,-11.4000){\makebox(0,0)[lt]{$\max h(A,t)+[t,2t]$}}%
%
\special{pn 4}%
\special{pa 3900 1080}%
\special{pa 3900 850}%
\special{fp}%
\special{sh 1}%
\special{pa 3900 850}%
\special{pa 3880 917}%
\special{pa 3900 903}%
\special{pa 3920 917}%
\special{pa 3900 850}%
\special{fp}%
%
\special{pn 20}%
\special{sh 1}%
\special{ar 810 800 10 10 0  6.28318530717959E+0000}%
\special{sh 1}%
\special{ar 870 800 10 10 0  6.28318530717959E+0000}%
\special{sh 1}%
\special{ar 1070 800 10 10 0  6.28318530717959E+0000}%
%
\special{pn 20}%
\special{pa 940 800}%
\special{pa 1030 800}%
\special{fp}%
%
\special{pn 20}%
\special{pa 1150 800}%
\special{pa 1260 800}%
\special{fp}%
%
\special{pn 8}%
\special{pa 400 800}%
\special{pa 4600 800}%
\special{dt 0.045}%
\special{pa 4600 800}%
\special{pa 4599 800}%
\special{dt 0.045}%
\put(6.2000,-11.4000){\makebox(0,0)[lt]{$\min h(A,t)-tg(A)$}}%
%
\special{pn 4}%
\special{pa 700 750}%
\special{pa 700 700}%
\special{fp}%
\special{pa 700 700}%
\special{pa 1300 700}%
\special{fp}%
\special{pa 1300 700}%
\special{pa 1300 750}%
\special{fp}%
\put(4.0000,-4.5000){\makebox(0,0)[lb]{$\min h(A,t)-[t,2t]$}}%
%
\special{pn 4}%
\special{pa 1000 500}%
\special{pa 1000 670}%
\special{fp}%
\special{sh 1}%
\special{pa 1000 670}%
\special{pa 1020 603}%
\special{pa 1000 617}%
\special{pa 980 603}%
\special{pa 1000 670}%
\special{fp}%
%
\special{pn 4}%
\special{pa 1020 1080}%
\special{pa 1020 850}%
\special{fp}%
\special{sh 1}%
\special{pa 1020 850}%
\special{pa 1000 917}%
\special{pa 1020 903}%
\special{pa 1040 917}%
\special{pa 1020 850}%
\special{fp}%
\end{picture}%
\end{center}
\end{figure}

For each $t > 0$,
 it is clear that $\phi(A,t) \notin \nwd(\R) \cup \perf(\R)$.
Since $h(A,t)$ contains an isolated point from $\R \setminus D$
 which remains to be isolated in $\phi(A,t)$,
 we see that $\phi(A,t) \notin \bcl(D)$.
Given $\phi(A,t)$ for $t>0$,
 we can reconstruct $t$ as the length of
 the interval $J \subs \phi(A,t)$ with $\max J = \max\phi(A,t)$.
Consequently, $g(A)$ can be reconstructed from $\phi(A,t)$.
Thus, $\phi$ is a strict deformation.

\smallskip
Next, we show the case $m > 1$.
To see that $\bcl(\Rm) \setminus \perf(\Rm)$ and
 $\bcl(\Rm) \setminus \bcl(D)$ are strictly homotopy dense in $\bcl_H(\Rm)$,
 we shall construct a strict deformation
 $\map\phi{\bcl_H(\Rm)\times[0,1]}{\bcl_H(\Rm)}$
 which omits $\perf(\Rm) \cup \bcl(D)$.
Recall $p : \Rm \to \R^{m-1}$ is the projection
 onto the first $m - 1$ coordinates.
Note that $\img pD$ is a dense $G_\delta$ set in $\R^{m-1}$
 and $\R^{m-1} \setminus \img pD$ is also dense in $\R^{m-1}$.
Let $e_m=\seq{0,0,\dots,0,1}\in\Rm$.

Since $\Rm \setminus (\img pD \times \R)$ is dense in $\Rm$,
 it follows from Lemma \ref{sdegfaqqfas}
 that $\Fin(\Rm) \setminus \bcl(\img pD \times \R)$
 is homotopy dense in $\bcl_H(\Rm)$.
Let $h$ be a homotopy of $\bcl(\Rm)$ which witnesses this,
 i.e., for $t > 0$, $h(A,t)$ is finite and $p[h(A,t)] \not\subs \img pD$.
Since $\bcl_H([3/5,2/3]) \topiso \Q$,
 we have an embedding $g : \bcl_H(\Rm) \to \bcl_H([3/5,2/3])$.
The desired $\phi$ can be defined as follows:
$$\phi(A,t) = h(A,t) + t\left(\bigcup_{i\in\omega}
 2^{-i}(g(A) \cup [3/4,1])e_m \cup \{2e_m\}\right).
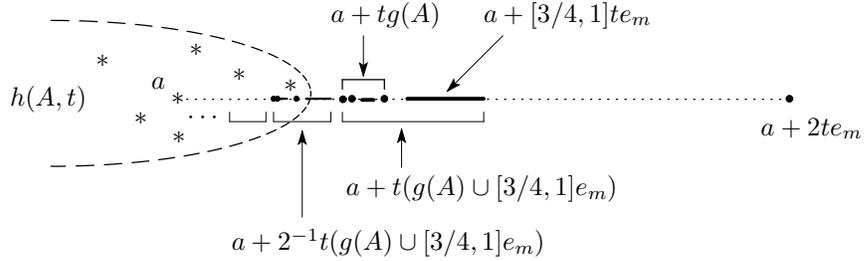
$$
\begin{figure}[hbtp]
\begin{center}\small
\unitlength 0.1in
\begin{picture}(40.83,12.20)(5.20,-14.80)
%
\special{pn 8}%
\special{pa 1400 803}%
\special{pa 4600 803}%
\special{dt 0.045}%
\special{pa 4600 803}%
\special{pa 4599 803}%
\special{dt 0.045}%
%
\special{pn 13}%
\special{sh 1}%
\special{ar 1900 800 10 10 0  6.28318530717959E+0000}%
%
\special{pn 20}%
\special{sh 1}%
\special{ar 2310 800 10 10 0  6.28318530717959E+0000}%
%
\special{pn 13}%
\special{pa 2070 800}%
\special{pa 2200 800}%
\special{fp}%
%
\special{pn 20}%
\special{pa 2363 806}%
\special{pa 2423 806}%
\special{fp}%
\put(12.6000,-7.5000){\makebox(0,0)[lb]{$a$}}%
%
\special{pn 20}%
\special{sh 1}%
\special{ar 4600 800 10 10 0  6.28318530717959E+0000}%
%
\special{pn 20}%
\special{sh 1}%
\special{ar 2263 803 10 10 0  6.28318530717959E+0000}%
%
\special{pn 4}%
\special{pa 2370 470}%
\special{pa 2370 680}%
\special{fp}%
\special{sh 1}%
\special{pa 2370 680}%
\special{pa 2390 613}%
\special{pa 2370 627}%
\special{pa 2350 613}%
\special{pa 2370 680}%
\special{fp}%
\put(22.0000,-4.3000){\makebox(0,0)[lb]{$a+tg(A)$}}%
%
\special{pn 4}%
\special{pa 3100 470}%
\special{pa 2810 760}%
\special{fp}%
\special{sh 1}%
\special{pa 2810 760}%
\special{pa 2871 727}%
\special{pa 2848 722}%
\special{pa 2843 699}%
\special{pa 2810 760}%
\special{fp}%
\put(30.0000,-4.3000){\makebox(0,0)[lb]{$a+[3/4,1]te_m$}}%
%
\special{pn 20}%
\special{sh 1}%
\special{ar 2480 800 10 10 0  6.28318530717959E+0000}%
\put(44.5000,-9.0000){\makebox(0,0)[lt]{$a+2te_m$}}%
%
\special{pn 20}%
\special{pa 3000 800}%
\special{pa 2600 800}%
\special{fp}%
%
\special{pn 8}%
\special{pa 2480 750}%
\special{pa 2480 700}%
\special{fp}%
\special{pa 2480 700}%
\special{pa 2260 700}%
\special{fp}%
\special{pa 2260 700}%
\special{pa 2260 750}%
\special{fp}%
%
\special{pn 4}%
\special{pa 2260 870}%
\special{pa 2260 920}%
\special{fp}%
\special{pa 2260 920}%
\special{pa 3000 920}%
\special{fp}%
\special{pa 3000 920}%
\special{pa 3000 870}%
\special{fp}%
%
\special{pn 4}%
\special{pa 2200 870}%
\special{pa 2200 920}%
\special{fp}%
\special{pa 2200 920}%
\special{pa 1900 920}%
\special{fp}%
\special{pa 1900 920}%
\special{pa 1900 870}%
\special{fp}%
%
\special{pn 13}%
\special{pa 1930 800}%
\special{pa 1970 800}%
\special{fp}%
%
\special{pn 13}%
\special{sh 1}%
\special{ar 2020 800 10 10 0  6.28318530717959E+0000}%
%
\special{pn 13}%
\special{sh 1}%
\special{ar 1920 800 10 10 0  6.28318530717959E+0000}%
%
\special{pn 4}%
\special{pa 2570 1160}%
\special{pa 2570 960}%
\special{fp}%
\special{sh 1}%
\special{pa 2570 960}%
\special{pa 2550 1027}%
\special{pa 2570 1013}%
\special{pa 2590 1027}%
\special{pa 2570 960}%
\special{fp}%
\put(22.8000,-12.0000){\makebox(0,0)[lt]{$a+t(g(A)\cup[3/4,1]e_m)$}}%
%
\special{pn 4}%
\special{pa 1860 870}%
\special{pa 1860 920}%
\special{fp}%
\special{pa 1860 920}%
\special{pa 1670 920}%
\special{fp}%
\special{pa 1670 920}%
\special{pa 1670 870}%
\special{fp}%
\put(15.4000,-9.0000){\makebox(0,0){$\cdots$}}%
\put(14.0000,-8.0000){\makebox(0,0){$*$}}%
\put(14.0500,-10.0200){\makebox(0,0){$*$}}%
\put(12.0500,-9.0200){\makebox(0,0){$*$}}%
\put(10.0500,-6.0200){\makebox(0,0){$*$}}%
\put(14.9500,-5.4200){\makebox(0,0){$*$}}%
%
\special{pn 8}%
\special{ar 735 772 1360 382  4.7123890 4.7812753}%
\special{ar 735 772 1360 382  4.8226071 4.8914935}%
\special{ar 735 772 1360 382  4.9328253 5.0017116}%
\special{ar 735 772 1360 382  5.0430434 5.1119297}%
\special{ar 735 772 1360 382  5.1532615 5.2221479}%
\special{ar 735 772 1360 382  5.2634797 5.3323660}%
\special{ar 735 772 1360 382  5.3736978 5.4425842}%
\special{ar 735 772 1360 382  5.4839160 5.5528023}%
\special{ar 735 772 1360 382  5.5941341 5.6630204}%
\special{ar 735 772 1360 382  5.7043522 5.7732386}%
\special{ar 735 772 1360 382  5.8145704 5.8834567}%
\special{ar 735 772 1360 382  5.9247885 5.9936749}%
\special{ar 735 772 1360 382  6.0350067 6.1038930}%
\special{ar 735 772 1360 382  6.1452248 6.2141111}%
\special{ar 735 772 1360 382  6.2554429 6.3243293}%
\special{ar 735 772 1360 382  6.3656611 6.4345474}%
\special{ar 735 772 1360 382  6.4758792 6.5447656}%
\special{ar 735 772 1360 382  6.5860974 6.6549837}%
\special{ar 735 772 1360 382  6.6963155 6.7652018}%
\special{ar 735 772 1360 382  6.8065336 6.8754200}%
\special{ar 735 772 1360 382  6.9167518 6.9856381}%
\special{ar 735 772 1360 382  7.0269699 7.0958563}%
\special{ar 735 772 1360 382  7.1371881 7.2060744}%
\special{ar 735 772 1360 382  7.2474062 7.3162925}%
\special{ar 735 772 1360 382  7.3576243 7.4265107}%
\special{ar 735 772 1360 382  7.4678425 7.5367288}%
\special{ar 735 772 1360 382  7.5780606 7.6469470}%
\special{ar 735 772 1360 382  7.6882788 7.7571651}%
\special{ar 735 772 1360 382  7.7984969 7.8539816}%
\put(5.2000,-8.7000){\makebox(0,0)[lb]{$h(A,t)$}}%
\put(16.8000,-14.8000){\makebox(0,0)[lt]{$a+2^{-1}t(g(A)\cup[3/4,1]e_m)$}}%
%
\special{pn 4}%
\special{pa 2050 1420}%
\special{pa 2050 970}%
\special{fp}%
\special{sh 1}%
\special{pa 2050 970}%
\special{pa 2030 1037}%
\special{pa 2050 1023}%
\special{pa 2070 1037}%
\special{pa 2050 970}%
\special{fp}%
\put(19.9000,-7.3000){\makebox(0,0){$*$}}%
\put(17.2000,-6.7000){\makebox(0,0){$*$}}%
\end{picture}%
\end{center}
\end{figure}

\medskip
For each $t > 0$,
 $\phi(A,t)$ has an isolated point because
 $\max \pr_m[\phi(A,t)]$ is attained by an isolated point of $\phi(A,t)$,
 where $\pr_m$ denotes the projection onto the $m$-th coordinate.
Hence, $\phi(A,t) \not\in \perf(\Rm)$.
Since $\img p{\varphi(A,t)} = \img p{h(A,t)}$ is finite
 and contains a point of $\R^{m-1} \setminus \img pD$,
 it follows that
 $\cl(\phi(A,t) \cap (\img pD \times \R)) \not= \phi(A,t)$,
 which means $\phi(A,t) \not\in \bcl(\img pD \times \R)$.

Given $\phi(A,t)$ for $t > 0$,
 we can find $t$ as the distance from
 $\max\pr_m[\phi(A,t)]$ to the interior of $\pr_m[\phi(A,t)]$.
Let $a_0 \in \phi(A,t)$ be such that
$$\pr_m(a_0) = \min\pr_m[\phi(A,t)] = \min\pr_m[h(A,t)].$$
Then, for sufficiently large $i$,
$$(a_0 + 2^{-i}t(g(A) \cup [3/4,1])e_m) \cap h(A,t) = \emptyset.$$
Thus, we can reconstruct $2^{-i}tg(A)$ and
 consequently also $g(A)$ from $\phi(A,t)$.
This shows that $\phi$ is a strict deformation.

For $\bcl(\Rm) \setminus \nwd(\Rm)$,
 we define a homotopy
 $\map\psi{\bcl_H(\Rm)\times[0,1]}{\bcl_H(\Rm)}$ as follows:
$$\psi(A,t) = h(A,t) + t\left(\bigcup_{i\in\omega}
 2^{-i}(g(A) \cup [3/4,1])e_m \cup \clbal(2e_m,1/2)\right).$$
In other wards,
 replacing the points $a + 2te_m \in \phi(A,t)$, $a \in h(A,t)$,
 by the closed balls
$$a + t\clbal(2e_m,1/2) = \clbal(a + 2te_m,t/2),\ a \in h(A,t),$$
 we can obtain $\psi(A,t)$ from $\phi(A,t)$.
Evidently $\psi$ omits $\nwd(\Rm)$.
Given $\psi(A,t)$ for $t > 0$,
 let $a_0 \in \psi(A,t)$ be such that
$$\pr_m(a_0) = \min\pr_m[\psi(A,t)] = \min\pr_m[h(A,t)].$$
Then we can get $t$ as the diameter of
 the ball $\clbal(a_0 + 2te_m,t/2)$
 (which is equal to $2/3$ of the distance from $a_0$ to this ball).
Now, by the same arguments as for $\phi$,
 we can reconstruct $g(A)$ from $\psi(A,t)$.
Thus, $\psi$ is a strict deformation.
\end{proof}

Let us note that the subspace $\uclb(\R)\cup\Fin(\R)$ is actually equal to the space $\poly(\R)$ consisting of all compact polyhedra in $\R$. It follows from the result of \cite{Sakai91} that the pair $\pair{\exp(\R)}{\poly(\R)}$ is homeomorphic to $\pair\Q{\Q_f}$.

\section{Nonseparable components of $\cld_H(\R)$}

In this section,
 we consider the space $\cld_H(\R)$ of all nonempty closed subsets of $\R$.
We shall also consider its natural subspaces,
 using the same notation as before,
 but having in mind the new setting.
For example,
 $\perf(\R)$ and $\nwd(\R)$ will denote the subspace of $\cld(\R)$
 consisting of all perfect closed subsets of $\R$
 and all closed sets with no interior points, respectively.
Now $\perf(\R) \cap \nwd(\R)$ consists of all nonempty closed
 (possibly unbounded) subsets of $\R$
 which have neither isolated points nor interior points.
In the new setting,
 we have
$$\Cantor(\R) = \perf(\R) \cap \nwd(\R) \cap \bcl(\R).$$

As shown in \cite[Proposition 7.2]{KuSaY},
 $\cld_H(\R)$ has $2^{\aleph_0}$ many components,
 $\bcl(\R)$ is the only separable one and
 any other component has weight $2^{\aleph_0}$.
The following is the main theorem in this section:

\begin{tw}\label{non-sep-compon}
Let\/ $\Ha$ be a nonseparable component of\/ $\cld_H(\R)$
 which does not contain\/ $\R$, $[0,+\infty)$, $(-\infty,0]$.
Then $\Ha \topiso \ell_2(2^{\aleph_0})$.
\end{tw}

We shall say that a set $A\subs\R$ {\em has infinite uniform gaps}
 if there are $\delta>0$ and pairwise disjoint open intervals $I_0,I_1,\dots$
 such that $\diam I_n \goe \delta$, $A \cap I_n = \emptyset$
 and $\bd I_n \subs A$ for every $\ntr$.
Define
$$\Vee = \setof{A\in\cld(\R)}{A\text{ has infinite uniform gaps }}.$$
Clearly, $\Vee$ is open in $\cld_H(\R)$
 and $\Vee \cap \bcl(\R) = \emptyset$.
For each $A \in \cld(\R) \setminus \bcl(\R)$ and $\varepsilon > 0$,
 let $D \subs A$ be a maximal $\varepsilon$-discrete subset.
Then $D \in \Vee$ and $d_H(A,D) \loe \varepsilon$
 because $D \subs A \subs \bal(D,\varepsilon)$.
Thus, $\Vee$ is dense in $\cld_H(\R)\setminus\bcl(\R)$.

If $\Ha$ is a nonseparable component of $\cld_H(\R)$
 and $\R,[0,+\infty),(-\infty,0]\notin \Ha$
 then $\Ha \subs \Vee$.
Indeed, each $A \in \Ha$ is unbounded and
 every component of $\R \setminus A$ is an open interval.
Let $\mathcal J$ be the set of all bounded component of $\R \setminus A$.
Assume that $\setof{\diam I}{I \in \mathcal J}$ is bounded.
When $A$ is bounded below (or bounded above),
 $d_H(A, [0,\infty)) < \infty$ (or $d_H(A, (-\infty,0]) < \infty$),
 which implies $[0,+\infty) \in \Ha$ (or $(-\infty,0] \in \Ha$).
When $A$ is not bounded below nor above,
 $d_H(A,\R) < \infty$,
 which implies $\R \in \Ha$.
Therefore, $\setof{\diam I}{I \in \mathcal J}$ is unbounded.
In particular, $A$ has infinite uniform gaps.

Due to Theorem A in \cite{KuSaY},
 every component of $\cld_H(\R)$ is an AR,
 hence it is contractible.
Since a contractible $\ell_2(2^{\aleph_0})$-manifold
 is homeomorphic to $\ell_2(2^{\aleph_0})$,
 Theorem \ref{non-sep-compon} above follows from the following theorem:

\begin{tw}
The open dense subset $\Vee$ of\/ $\cld_H(\R)$
 is an $\ell_2(2^{\aleph_0})$-manifold.
\end{tw}

\begin{proof}
It suffices to show that
 each $A_0 \in \Vee$ has an open neighborhood $\Yu\subs\Vee$
 which is an $\ell_2(2^{\aleph_0})$-manifold.
In this case,
 $\Yu$ is a completely metrizable ANR
 because it is an open set in a completely metrizable ANR $\cld_H(\R)$.
Due to Toru{\'n}czyk characterization of $\ell_2(2^{\aleph_0})$-manifold
 \cite{To81} (cf.\ \cite{To85}),
 we have to show that $\Yu$ has the following two properties:
\begin{romanenume}
\item
 For each maps $f : [0,1]^n \times 2^\omega \to \Yu$
 and $\alpha : \Yu \to (0,1)$,
 there exists a map $g : [0,1]^n \times 2^\omega \to \Yu$ such that
 $d_H(g(z),f(z)) < \alpha(f(z))$ for each $z \in [0,1]^n \times 2^\omega$
 and $\{g[[0,1]^n \times \{x\}] : x \in 2^\omega\}$ is discrete in $\Yu$;
\item
 For any finite-dimensional simplicial complexes $K_n$, $n \in \omega$,
 with $\card K_n \loe 2^{\aleph_0}$,
 for every maps $f : \bigoplus_{n\in\omega} |K_n| \to \Yu$
 and $\alpha : \Yu \to (0,1)$,
 there exists a map $g : \bigoplus_{n\in\omega} |K_n| \to \Yu$ such that
 $d_H(g(z),f(z))$ $ < \alpha(f(z))$ for each $z \in \bigoplus_{n\in\omega} |K_n|$
 and $\{g[|K_n|] : n \in \omega\}$ is discrete in $\Yu$.
\end{romanenume}
In the above, $2^\omega$ is the discrete space of all functions
 of $\omega$ to $2 = \{0,1\}$.
To this end,
 it suffices to prove the following:
\begin{itemize}
\item
 For each map $\alpha : \Yu \to (0,1)$,
 there exist maps $f_x : \Yu \to \Yu$, $x \in 2^\omega$,
 such that $d_H(f_x(A),A) < \alpha(A)$ for every $A \in \Yu$
 and $\{f_x[\Yu] : x \in 2^\omega\}$ is discrete.
\end{itemize}

Fix $A_0 \in \Vee$ and choose open intervals $I_0,I_1,\dots$
 such that $\diam I_n \goe \delta$, $A_0 \cap I_n = \emptyset$
 and $\bd I_n \subs A_0$ (i.e., $\inf I_n,\ \sup I_n \in A_0$)
 for every $\ntr$.
Taking a subsequence if necessary,
 we may assume that either $\sup I_n < \inf I_{n+1}$ for every $\ntr$
 or $\inf I_n > \sup I_{n+1}$ for every $\ntr$.
Because of similarity,
 we may assume that the first possibility occurs.

Choose intervals $[a_n,b_n] \subs I_n$, $\ntr$,
 so that $b_n - a_n > \delta/4$,
\begin{gather*}
\inf_{n\in\omega}\dist(a_n,\R \setminus I_n)
 = \inf_{n\in\omega}(a_n - \inf I_n) > \delta/4 \text{ and}\\
\inf_{n\in\omega}\dist(b_n,\R \setminus I_n)
 = \inf_{n\in\omega}(\sup I_n - b_n) > \delta/4.
\end{gather*}
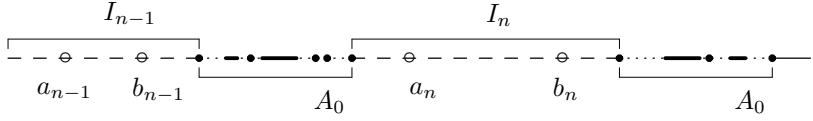
\begin{figure}[hbtp]
\begin{center}\small
\unitlength 0.1in
\begin{picture}(42.15,5.10)(3.85,-9.70)
%
\special{pn 8}%
\special{pa 400 800}%
\special{pa 1400 800}%
\special{da 0.070}%
%
\special{pn 8}%
\special{pa 2200 800}%
\special{pa 3600 800}%
\special{da 0.070}%
%
\special{pn 8}%
\special{pa 1400 800}%
\special{pa 2200 800}%
\special{dt 0.045}%
\special{pa 2200 800}%
\special{pa 2199 800}%
\special{dt 0.045}%
%
\special{pn 8}%
\special{pa 3600 800}%
\special{pa 4400 800}%
\special{dt 0.045}%
\special{pa 4400 800}%
\special{pa 4399 800}%
\special{dt 0.045}%
\put(9.0000,-6.3000){\makebox(0,0)[lb]{$I_{n-1}$}}%
\put(29.0000,-6.3000){\makebox(0,0)[lb]{$I_n$}}%
%
\special{pn 20}%
\special{sh 1}%
\special{ar 1400 800 10 10 0  6.28318530717959E+0000}%
%
\special{pn 20}%
\special{sh 1}%
\special{ar 2200 800 10 10 0  6.28318530717959E+0000}%
%
\special{pn 20}%
\special{sh 1}%
\special{ar 3600 800 10 10 0  6.28318530717959E+0000}%
%
\special{pn 20}%
\special{sh 1}%
\special{ar 1670 800 10 10 0  6.28318530717959E+0000}%
%
\special{pn 20}%
\special{sh 1}%
\special{ar 2070 800 10 10 0  6.28318530717959E+0000}%
%
\special{pn 20}%
\special{sh 1}%
\special{ar 2010 800 10 10 0  6.28318530717959E+0000}%
%
\special{pn 20}%
\special{pa 1730 800}%
\special{pa 1910 800}%
\special{fp}%
%
\special{pn 20}%
\special{pa 3840 800}%
\special{pa 4020 800}%
\special{fp}%
%
\special{pn 20}%
\special{sh 1}%
\special{ar 4070 800 10 10 0  6.28318530717959E+0000}%
%
\special{pn 4}%
\special{pa 1400 850}%
\special{pa 1400 900}%
\special{fp}%
\special{pa 2200 900}%
\special{pa 2200 850}%
\special{fp}%
\special{pa 1400 900}%
\special{pa 2200 900}%
\special{fp}%
\put(20.0000,-9.7000){\makebox(0,0)[lt]{$A_0$}}%
%
\special{pn 4}%
\special{pa 3600 850}%
\special{pa 3600 900}%
\special{fp}%
\special{pa 3600 900}%
\special{pa 4400 900}%
\special{fp}%
\special{pa 4400 900}%
\special{pa 4400 850}%
\special{fp}%
\put(42.0000,-9.7000){\makebox(0,0)[lt]{$A_0$}}%
\put(25.0000,-9.3000){\makebox(0,0)[lt]{$a_n$}}%
\put(32.5000,-9.0000){\makebox(0,0)[lt]{$b_n$}}%
\put(10.5000,-9.0000){\makebox(0,0)[lt]{$b_{n-1}$}}%
\put(5.5000,-9.3000){\makebox(0,0)[lt]{$a_{n-1}$}}%
%
\special{pn 20}%
\special{sh 1}%
\special{ar 4400 800 10 10 0  6.28318530717959E+0000}%
%
\special{pn 20}%
\special{pa 4180 800}%
\special{pa 4260 800}%
\special{fp}%
%
\special{pn 20}%
\special{pa 1540 800}%
\special{pa 1600 800}%
\special{fp}%
%
\special{pn 8}%
\special{pa 4400 800}%
\special{pa 4600 800}%
\special{fp}%
%
\special{pn 4}%
\special{pa 400 750}%
\special{pa 400 700}%
\special{fp}%
\special{pa 400 700}%
\special{pa 1400 700}%
\special{fp}%
\special{pa 1400 700}%
\special{pa 1400 750}%
\special{fp}%
%
\special{pn 4}%
\special{pa 2200 750}%
\special{pa 2200 700}%
\special{fp}%
\special{pa 2200 700}%
\special{pa 3600 700}%
\special{fp}%
\special{pa 3600 700}%
\special{pa 3600 750}%
\special{fp}%
\put(25.0000,-8.0000){\makebox(0,0){$\circ$}}%
\put(33.0000,-8.0000){\makebox(0,0){$\circ$}}%
\put(11.0000,-8.0000){\makebox(0,0){$\circ$}}%
\put(7.0000,-8.0000){\makebox(0,0){$\circ$}}%
\end{picture}%
\end{center}
\end{figure}

\bigskip\noindent
Observe that if $A \in \cld_H(\R)$ and $d_H(A,A_0) < \delta/4$
 then $A \cap (b_{n-1},a_n) \not= \emptyset$
 for every $n \in \omega$,
 where $b_{-1} = -\infty$.
For each $A \in \cld_H(\R)$ with $d_H(A,A_0) < \delta/4$,
 we can define
$$r_n(A) = \max(A \cap (b_{n-1},a_n)),\ \ntr.$$
For each $A, A' \in \cld_H(\R)$
 with $d_H(A,A_0), d_H(A',A_0) < \delta/4$,
 we have
$$|r_n(A) - r_n(A')| \leqslant d_H(A,A').$$
Indeed, without loss of generality, we may assume that $r_n(A) < r_n(A')$.
Then, the open interval $(r_n(A),b_n)$ contains no points of $A$
 and $r_n(A') \in (r_n(A),b_n)$.
Since $b_n - r_n(A') > \delta/2$ and
$$r_n(A') - r_n(A)
 \leqslant |r_n(A') - r_n(A_0)| + |r_n(A) - r_n(A_0)| < \delta/2,$$
 we have $|r_n(A') - r_n(A)| \leqslant d_H(A,A')$.
Then, it follows that
\begin{align*}
\inf_{n\in\omega}(a_n - r_n(A)) - d_H(A,A')
 &\leqslant \inf_{n\in\omega}(a_n - r_n(A')) \\
 &\leqslant \inf_{n\in\omega}(a_n - r_n(A)) + d_H(A,A').
\end{align*}
This means that $A \mapsto \inf_{n\in\omega}(a_n - r_n(A))$ is continuous.
Since $r_n(A_0) = \inf I_n$,
 we have $\inf_{n\in\omega}(a_n - r_n(A_0)) > \delta/4$.
Thus, $A_0$ has the following open neighborhood:
$$\Yu = \setof{A\in\cld_H(\R)}{d_H(A,A_0) < \delta/4,\
 \inf_{n\in\omega}(a_n - r_n(A)) > \delta/4} \subs \Vee.$$

Now, for each map $\map \alpha\Yu{(0,1)}$,
 we define a map $\beta : \Yu \to (0,1)$ as follows:
$$\beta(A) = \min\big\{\tfrac12\alpha(A),\ \tfrac14\delta - d_H(A,A_0),\
 \inf_{n\in\omega}(a_n - r_n(A)) - \tfrac14\delta\big\}.$$
Given a sequence $x = (x(n))_{\ntr} \in 2^\omega$, let
$$f_x(A) = A \cup \bigcup_{\ntr}\big(r_n(A) +
 \big([0,\tfrac12\beta(A)] \cup \{\beta(A)\cdot x(n)\}\big)\big).$$
\begin{figure}[hbtp]
\begin{center}\small
\unitlength 0.1in
\begin{picture}(29.15,7.30)(4.85,-13.00)
%
\special{pn 4}%
\special{pa 1440 1050}%
\special{pa 1440 1100}%
\special{fp}%
\special{pa 2240 1100}%
\special{pa 2240 1050}%
\special{fp}%
\special{pa 1440 1100}%
\special{pa 2240 1100}%
\special{fp}%
\put(7.0000,-10.9000){\makebox(0,0)[lt]{$b_{n-1}$}}%
%
\special{pn 8}%
\special{pa 600 1000}%
\special{pa 3400 1000}%
\special{dt 0.045}%
\special{pa 3400 1000}%
\special{pa 3399 1000}%
\special{dt 0.045}%
%
\special{pn 20}%
\special{sh 1}%
\special{ar 1640 1000 10 10 0  6.28318530717959E+0000}%
%
\special{pn 20}%
\special{sh 1}%
\special{ar 2010 1003 10 10 0  6.28318530717959E+0000}%
%
\special{pn 20}%
\special{pa 1730 1003}%
\special{pa 1910 1003}%
\special{fp}%
%
\special{pn 20}%
\special{pa 1540 1003}%
\special{pa 1600 1003}%
\special{fp}%
%
\special{pn 8}%
\special{pa 2240 980}%
\special{pa 2420 980}%
\special{fp}%
\put(30.5000,-11.0000){\makebox(0,0)[lt]{$a_n$}}%
\put(14.5000,-12.0000){\makebox(0,0)[lt]{$A\cap(b_{n-1},a_n)$}}%
%
\special{pn 20}%
\special{sh 1}%
\special{ar 1440 1000 10 10 0  6.28318530717959E+0000}%
%
\special{pn 4}%
\special{pa 2020 770}%
\special{pa 2220 970}%
\special{fp}%
\special{sh 1}%
\special{pa 2220 970}%
\special{pa 2187 909}%
\special{pa 2182 932}%
\special{pa 2159 937}%
\special{pa 2220 970}%
\special{fp}%
\put(20.5000,-7.4000){\makebox(0,0)[rb]{$r_n(A)$}}%
%
\special{pn 4}%
\special{pa 2640 1260}%
\special{pa 2640 1060}%
\special{fp}%
\special{sh 1}%
\special{pa 2640 1060}%
\special{pa 2620 1127}%
\special{pa 2640 1113}%
\special{pa 2660 1127}%
\special{pa 2640 1060}%
\special{fp}%
\put(24.0000,-13.0000){\makebox(0,0)[lt]{$r_n(A) + \beta(A)$}}%
%
\special{pn 4}%
\special{pa 2420 760}%
\special{pa 2320 960}%
\special{fp}%
\special{sh 1}%
\special{pa 2320 960}%
\special{pa 2368 909}%
\special{pa 2344 912}%
\special{pa 2332 891}%
\special{pa 2320 960}%
\special{fp}%
\put(24.2000,-7.4000){\makebox(0,0)[lb]{$r_n(A) + [0,\frac12\beta(A)]$}}%
%
\special{pn 20}%
\special{sh 1}%
\special{ar 2200 1000 10 10 0  6.28318530717959E+0000}%
%
\special{pn 20}%
\special{sh 1}%
\special{ar 2150 1000 10 10 0  6.28318530717959E+0000}%
\put(26.4000,-10.0000){\makebox(0,0){$*$}}%
%
\special{pn 8}%
\special{pa 2240 1020}%
\special{pa 2420 1020}%
\special{fp}%
\put(8.0000,-10.0000){\makebox(0,0){$\circ$}}%
\put(30.8000,-10.0000){\makebox(0,0){$\circ$}}%
\end{picture}%
\end{center}
\end{figure}

\noindent
This defines a map $\map{f_x}\Yu\Yu$ which is $\alpha$-close to $\id$.
We claim that
 if $x \not= y \in 2^\omega$
 then
$$d_H(f_x(A),f_y(A')) \goe \min\big\{\tfrac14\beta(A),\tfrac14\beta(A')\big\}
 \text{ for every $A, A' \in \Yu$.}$$
Indeed, assume that $x(n) = 1$, $y(n) = 0$
 and let $s = \min\{\tfrac14\beta(A),\tfrac14\beta(A')\}$.
Then
\begin{enumerate}
\item
 $\max(f_x(A) \cap (b_{n-1},a_n)) = r_n(A)+\beta(A)$;
\item
 $f_x(A)$ has no points in the open interval
 $(r_n(A)+\tfrac12\beta(A), r_n(A)+\beta(A))$;
\item
 $\max(f_y(A') \cap (b_{n-1},a_n)) = r_n(A')+\tfrac12\beta(A')$;
\item
 $[r_n(A'),r_n(A')+\beta(A')/2] \subs f_y(A')$.
\end{enumerate}
In case $r_n(A')+\tfrac12\beta(A') \goe r_n(A)+\beta(A)+s$ or
 $r_n(A')+\tfrac12\beta(A') \loe r_n(A)+\beta(A)-s$,
 we have
$$d_H(f_x(A) \cap (b_{n-1},a_n),f_y(A') \cap (b_{n-1},a_n)) \goe s.$$
In case $r_n(A)+\beta(A)-s < r_n(A')+\tfrac12\beta(A') \loe r_n(A)+\beta(A)+s$,
 since $2s \loe \tfrac12\beta(A')$,
 we have $r_n(A') < r_n(A)+\beta(A)-s$,
 hence $r_n(A)+\beta(A)-s \in f_y(A')$.
Thus, it follows that
$$d_H(f_x(A) \cap (b_{n-1},a_n),f_y(A') \cap (b_{n-1},a_n)) \goe s.$$

Finally, we show that
 $\{f_x[\Yu] : x \in 2^\omega\}$ is a discrete collection of $\Yu$.
If not, we have $A$, $A_i \in \Yu$ and $x_i \in 2^\omega$, $i \in \omega$,
 such that $x_i \not= x_j$ if $i \not= j$,
 and $f_{x_i}(A_i) \to A$ ($i \to \infty$).
Then $c = \inf_{i\in\omega}\beta(A_i) = 0$.
Indeed, otherwise we could find $i < j$ such that
 $$d_H(f_{x_i}(A_i),A),\ d_H(f_{x_j}(A_j),A) < c/10$$ and
 $\beta(A_i),\ \beta(A_j) > 4c/5$.
It follows that $d_H(f_{x_i}(A_i),f_{x_j}(A_j)) < c/5$, but
$$d_H(f_{x_i}(A_i),f_{x_j}(A_j)) \goe \min\{\beta(A)/4,\beta(A')/4\}
 > c/5,$$
 which is a contradiction.
Thus, $\inf_{i\in\omega}\beta(A_i) = 0$.
Taking a subsequence,
 we may assume that $\lim_{i\to\infty}\beta(A_i) = 0$.
Then $A_i \to A$ ($i \to \infty$)
 because $d_H(f_{x_i}(A_i),A_i) \loe \beta(A_i)$.
It follows that $\beta(A) = 0$,
 which is a contradiction.
This completes the proof.
\end{proof}

Let $\Dee(X)$ be the subspace of $\cld_H(X)$ consisting of
 all discrete sets in $X$.
It follows from the result of \cite{BaVo} that
 $\Dee(X)$ is homotopy dense in $\cld_H(X)$
 for every almost convex metric space $X$.
By the same proof,
 Lemma \ref{sdegfaqqfas} can be extended to $\cld_H(\Rm)$.

\begin{prop}\label{h-dense-D}
Assume $D\subs \Rm$ is such that $\Rm\setminus D$ is dense.
Then $\Dee(\Rm)\setminus\cld(D)$ is homotopy dense in $\cld_H(\Rm)$.
\end{prop}

Now, we consider the subspaces $\En(\R)$, $\nwd(\R)$, $\perf(\R)$
 and $\cld(\R\setminus\Qyu)$ of $\cld_H(\R)$.
Similarly to $\bcl_H(\R)$,
 the following can be shown:

\begin{prop}\label{non-sep-Z_sig}
The sets $\cld(\R) \setminus \En(\R)$, $\cld(\R) \setminus \nwd(\R)$,
 $\cld(\R) \setminus \perf(\R)$ and $\cld(\R) \setminus \cld(\R\setminus\Qyu)$
 are \Zsig-sets in the space $\cld_H(\R)$.
\end{prop}

Due to Negligibility Theorem (\cite{AHW}, \cite{Cut})
 if $M$ is an $\ell_2(2^{\aleph_0})$-manifold
 and $A$ is a \Zsig-set in $M$ then $M \setminus A \topiso M$.
Thus, combining Proposition \ref{non-sep-Z_sig}
 and Theorem \ref{non-sep-compon},
 we have the following:

\begin{wn}
Let\/ $\Ha$ be a nonseparable component of\/ $\cld_H(\R)$
 which does not contain\/ $\R$, $[0,+\infty)$, $(-\infty,0]$.
Then $\Ha \cap \En(\R)$, $\Ha \cap \nwd(\R)$,
 $\Ha \cap \perf(\R)$ and\/ $\Ha \cap \cld(\R\setminus\Qyu)$
 are homeomorphic to $\ell_2(2^{\aleph_0})$.
\end{wn}

\section{Open problems}

The following questions are left open.

\begin{question}
In case $m > 1$,
 under the only assumption that $D \subs \Rm$ is a dense $G_\delta$ set
 and $\Rm \setminus D$ is also dense in $\Rm$,
 is the pair $\pair{\bcl(\R^m)}{\bcl(D)}$
 homeomorphic to $\pair\Qmp\pseudointmp$?
In particular,
 is the pair $\pair{\bcl(\R^m)}{\bcl(\nu^m_{m-1})}$
 homeomorphic to $\pair\Qmp\pseudointmp$?
\end{question}

\begin{question}
Does Theorem \ref{non-sep-compon} hold
 even if $\Ha$ contains $\R$, $[0,\infty)$ or $(-\infty,0]$?
\end{question}

\begin{question}
For $m > 1$,
 is $\cld_H(\Rm) \setminus \bcl(\Rm)$ an $\ell_2(2^{\aleph_0})$-manifold?
\end{question}

\section{Appendix}

For the convenience of readers,
we give short and straightforward proofs of Propositions \ref{rthsdpio} and \ref{owteepgsdgfa}.

\begin{prop}[\ref{rthsdpio}]
If $\pair Xd$ is $\sigma$-compact
 then the space $\pair{\bcl(X)}{d_H}$ is $F_{\sigma\delta}$
 in its completion $\pair{\bcl(\til X)}{d_H}$.
\end{prop}

\begin{proof}
Fix a countable open base $\setof{U_n}{\ntr}$ for $\til X$. Since $U_n\cap X$ is $F_\sigma$, we have $U_n\cap X=\bigcup_{k\in\nat}K_k^n$, where each $K_k^n$ is compact. Observe that, by compactness, the sets $(\til X\setminus K_k^n)^+$ are open in the Hausdorff metric topology. We claim that
$$\bcl(\til X)\setminus \bcl(X)=\bigcup_{\ntr}\Bigl(U_n^-\cap\bigcap_{k\in\nat}(\til X\setminus K_k^n)^+\Bigr),$$
which shows that $\bcl(\til X)\setminus \bcl(X)$ is a countable union of $G_\delta$ sets. This is what we want to prove.

Assume $A\in \bcl(\til X)\setminus \bcl(X)$, that is, $A\ne \cl_{\til X}(A\cap X)$. Then there is $\ntr$ such that $U_n\cap A\nnempty$ and $U_n\cap A\cap X=\emptyset$, which means that $A\in U_n^-$ and $A\in (\til X\setminus K_k^n)^+$ for every $k\in \nat$. Conversely, if $A\in U_n^-\cap \bigcap_{k\in\nat}(\til X\setminus K_k^n)^+$ then $U_n\cap A\nnempty$ and $U_n\cap A\cap X=\emptyset$, so $A\ne\cl_{\til X}(A\cap X)$.
\end{proof}

\begin{prop}[\ref{owteepgsdgfa}]
If $\pair Xd$ is Polish then the space $\pair{\bcl(X)}{d_H}$ is $G_\delta$ in its completion $\pair{\bcl(\til X)}{d_H}$.
\end{prop}

\begin{proof}
Let $\setof{W_n}{\ntr}$ be a family of open subsets of $\til X$ such that $X=\bigcap_{\ntr}W_n$. Fix a countable open base $\setof{V_n}{\ntr}$ for $\til X$. We claim that
\begin{equation}
\bcl(\til X)\setminus \bcl(X)=\bigcup_{\ntr}\bigcup_{k\in\nat}\Bigl(V_n^-\setminus(V_n\cap W_k)^-\Bigr).\tag{$*$}
\end{equation}
As $V^-$ is open in the metric space $\pair{\bcl(\til X,d)}{d_H}$ whenever $V\subs\til X$ is open, it follows that $V_n^-$ is $F_\sigma$ and therefore the set on the right-hand side of ($*$) is $F_\sigma$ in $\bcl_H(\til X)$. It remains to prove ($*$).

If $A\in V_n^-\setminus(V_n\cap W_k)^-$ then we have $x\in V_n\cap A$. Since $V_n\cap (A\cap X)=\emptyset$, it follows that $x\notin \cl_{\til X}(A\cap X)$. Thus $A\notin \bcl(X)$.
Now assume $A\in \bcl(\til X)\setminus \bcl(X)$, that is, $A\ne\cl_{\til X}(A\cap X)$. Then there exists an open set $U\subs \til X$ such that $U\cap A\nnempty$ and $U\cap A\cap X=\emptyset$. Hence $\bigcap_{k\in\nat}A\cap U\cap W_k=\emptyset$. Note that $A\cap U$ is a Baire space because of the completeness of $\pair{\til X}d$. Thus, by the Baire Category Theorem, there exists $k\in\nat$ such that $A\cap U\cap W_k$ is not dense in $A\cap U$. Find a basic open set $V_n\subs U$ such that $V_n\cap A\nnempty$ and $V_n\cap A\cap W_k=\emptyset$. Then $A\in V_n^-\setminus(V_n\cap W_k)^-$.
\end{proof}

Let $\borel(X)$ denote the Borel field on a topological space $X$.
Given $\mathfrak H \subs \cld(X)$,
 the {\em Effros $\sigma$-algebra\/} $\effros(\mathfrak H)$ is
 the $\sigma$-algebra generated by
$$\setof{U^-\cap\mathfrak H}{\text{$U$ is open in $X$}}.$$
It is well known that $\effros(\exp(X)) = \borel(\exp(X))$
 for every separable metric space $X$
 (see \cite[Theorem 6.5.15]{Beer}).\footnote
	{$\effros(\cld(X)) = \borel(\cld_H(X))$
	for every totally bounded separable metric space $X$
	(cf.\ \cite[Hess' Theorem 6.5.14 with Theorem 3.2.3]{Beer}).}
Whenever $X$ is a separable metric space
 in which every bounded set is totally bounded,
 we can regard $\bcl_H(X) \subs \exp(\tilde{X})$
 by the identification as in \S\ref{borelclasses},
 where $\tilde{X}$ is the completion of $X$.
Then, we have not only $\effros(\bcl(X)) = \borel(\bcl_H(X))$
 but also $\effros(\mathfrak H) = \borel(\mathfrak H)$
 for $\mathfrak H \subs \bcl_H(X)$.
This implies that $\effros(\mathfrak H)$ is standard
 if $\mathfrak H$ is absolutely Borel (cf.\ \cite[12.B]{Ke}).
The results in \S\ref{borelclasses} provide such hyperspaces $\mathfrak H$.

In relation to the results above,
 we can prove the following:

\begin{prop}\label{weptjwpf}
Let $X=\pair Xd$ be an analytic metric space
 in which bounded sets are totally bounded.
Then, the space $\bcl_H(X)$ is analytic.
\end{prop}

\begin{proof}
The completion $\pair{\til X}d$ of $\pair Xd$ is a Polish space in which closed bounded sets are compact. Then $\bcl_H(\til X,d)=\exp(\til X)$ is Polish. Fix a countable open base $\ciag U$ for $\til X$. Since $X$ is analytic, there exists a tree $\setof{X_s}{s\in\nat^{<\nat}}$ of closed subsets of $\til X$ such that $X=\bigcup_{f\in\nat^\nat}\bigcap_{\ntr}X_{f\rest n}$, which is the result of the Suslin operation on the family $\setof{X_s}{s\in\nat^{<\nat}}$ (e.g.\ see \cite[Lemma 11.7]{Jech}). We may assume that $X_s\sups X_t$ whenever $s\subs t$. Let $W_s=\bal(X_s,2^{-|s|})$, where $|s|$ denotes the length of the sequence $s$. Then $\cl W_s\sups \cl W_t$ whenever $s\subs t$. Moreover, $\bigcap_{\ntr}X_{f\rest n}=\bigcap_{\ntr}\cl W_{f\rest n}$ for each $f\in\nat^\nat$. We claim that
\begin{equation}
\bcl(X,d)=\bigcap_{k\in\nat}\bigcup_{f\in\nat^\nat}\bigcap_{\ntr}\Bigl((\bcl(\til X,d)\setminus U_k^-)\cup(U_k\cap W_{f\rest n})^-\Bigr),\tag{$\sharp$}
\end{equation}
where, as usual, we regard $\bcl(X,d)\subs\bcl(\til X,d)$, via the embedding $A\mapsto \cl_{\til X}A$. The above formula ($\sharp$) shows that $\bcl(X,d)$ can be obtained from $\bcl(\til X,d)$ by using the Suslin operation and countable intersection, which shows that it is analytic. It remains to prove ($\sharp$).

Fix $A\in\bcl(\til X,d)\setminus \bcl(X,d)$. Then $A\ne\cl(A\cap X)$ and hence there exists $k\in\nat$ such that $A\in U_k^-$ and $\cl U_k\cap A\cap X=\emptyset$. Then $A \notin \bcl(\til X,d)\setminus U_k^-$. For each $f\in\nat^\nat$, we have
$$A\cap\cl U_k\cap\bigcap_{\ntr}\cl W_{f\rest n}
=A\cap\cl U_k\cap\bigcap_{\ntr}X_{f\rest n}=\emptyset.$$
By compactness, there is $\ntr$ such that $A\cap\cl U_k\cap \cl W_{f\rest n}=\emptyset$, hence $A\notin(U_k\cap W_{f\rest n})^-$.

Now assume that $A\in\bcl(\til X,d)$ does not belong to the right-hand side of ($\sharp$), that is, there exists $k\in\nat$ such that $A\in U_k^-$ and for every $f\in\nat^\nat$ there is $\ntr$ with $A\notin(U_k\cap W_{f\rest n})^-$. In particular, $A\cap U_k\cap\bigcap_{\ntr}X_{f\rest n}=\emptyset$ for every $f\in\nat^\nat$ and consequently $U_k\cap A\cap X=\emptyset$. On the other hand, $A\cap U_k\nnempty$. Thus it follows that $A\ne \cl_{\til X}(A\cap X)$, which means $A\notin\bcl(X,d)$.
\end{proof}

\end{document}